\documentclass[a4paper,final,oneside,notitlepage,12pt]{article}

\usepackage{amssymb}
\usepackage{amsthm}
\usepackage{amsmath}
\usepackage[all]{xy}
\usepackage{graphicx}
\usepackage{mathrsfs}
\usepackage[bottom]{footmisc}
\usepackage[margin=2cm,font=normalsize,labelfont=bf]{caption}
\usepackage[normalem]{ulem}
\usepackage{verbatim}
\usepackage{amstext}
\usepackage{xstring}

\makeatletter

\gdef\@ntitle{\@title}
\def\subtitle#1{\gdef\@subtitle{#1}}
\def\@subtitle{}
\def\adress#1{\gdef\@adress{#1}}
\def\@adress{}
\def\preprint#1{\gdef\@preprint{#1}}
\def\@preprint{}
\def\keywords#1{\gdef\@keywords{#1}}
\def\@keywords{}
\def\email#1{\gdef\@email{#1}}
\def\@email{}

\def\refname{References}

\newlength{\myparskip}
\newlength{\myproofparskip}
\def\href#1#2{#2}

\def\kohyp{
  \usepackage{hyperref}
  \hypersetup{
    linktocpage = true,
    pdftitle = {\@title},
    pdfauthor = {\@author},
    pdfkeywords = {\@keywords},    
    bookmarksopen = true,
    bookmarksopenlevel = 1
  }}  
\def\showkeywords{\begin{flushleft}\footnotesize\textbf{Keywords}: \@keywords.\end{flushleft}}

\newcounter{mythm}[subsection]

\def\setsecnumdepth#1{
  \setcounter{secnumdepth}{#1}
  \setcounter{mythm}{0}
  \ifnum \c@secnumdepth >0
    \ifnum \c@secnumdepth >1
      \def\themythm{\thesubsection.\arabic{mythm}}
      \numberwithin{equation}{subsection}
      \renewcommand\theequation{\thesubsection.\arabic{equation}}
    \else
      \def\themythm{\thesection.\arabic{mythm}}
      \numberwithin{equation}{section}
      \renewcommand\theequation{\thesection.\arabic{equation}}
    \fi
  \else
    \def\themythm{\arabic{mythm}}
  \fi}

\newtheorem{theorem}[mythm]{Theorem}
\newtheorem{definition}[mythm]{Definition}
\newtheorem{rem}[mythm]{Remark}
\newtheorem{corollary}[mythm]{Corollary}
\newtheorem{exa}[mythm]{Example}
\newtheorem{proposition}[mythm]{Proposition}
\newtheorem{lemma}[mythm]{Lemma}

\newenvironment{remark}{\begin{rem}\normalfont\setlength{\parskip}{\myproofparskip}}{\setlength{\parskip}{\myparskip}\end{rem}}

\renewenvironment{proof}{Proof.\setlength{\parskip}{\myproofparskip}}{\hfill{$\square$}\\\setlength{\parskip}{\myparskip}}
\newenvironment{proofblank}[1]{#1.\setlength{\parskip}{\myproofparskip}}{\hfill{$\square$}\\}

\def\Z {\mathbb{Z}}

\def\R {\mathbb{R}}

\def\im{\mathrm{i}}
\def\id{\mathrm{id}}
\def\h {\mathrm{H}}
\def\babla{\blacktriangledown}
\def\trivlin{\mathbf{I}}
\def\quot#1{``#1''}

\def\tabularheight#1{\renewcommand\arraystretch{#1}}
\tabularheight{1.4}

\def\ttimes#1{\!\!\times_{#1}\!\!}

\renewcommand{\varepsilon}{\epsilon}
\renewcommand{\emph}[1]{\def\reserved@a{it}\ifx\f@shape\reserved@a\uline{#1}\else\textit{#1}\fi}

\def\bigset#1#2{\left\lbrace\;\begin{minipage}[c]{#1}\begin{center}#2\end{center}\end{minipage}\;\right\rbrace}

\newcommand\erf[1]{(\ref{#1})}

\newlength{\myl}

\newcommand{\ueins}{{\mathrm{U}}(1)}

\newcommand{\spin}[1]{{\mathrm{Spin}}(#1)}

\newcommand{\so}[1]{{\mathrm{SO}}(#1)}
\newcommand{\str}[1]{{\mathrm{String}}(#1)}

\def\triv#1{\mathcal{T}\!\!riv(#1)}
\def\trivcon#1{\mathcal{T}\!\!riv^{\nabla\!}(#1)}

\def\buntech#1#2{\mathcal{B}\hspace{-0.025cm}un^{#2}_{\!#1}}

\def\ubuncon#1{\buntech\relax\relax{\!\nabla\!}(#1)}
\def\grbtech#1{\mathcal{G}\hspace{-0.03cm}r\hspace{-0.02cm}b_{\!#1}}
\def\grb#1#2{\grbtech#1(#2)}
\def\grbcon#1#2{\grbtech#1^{\nabla\!}(#2)}
\def\ugrb#1{\grb{\,}{#1}}
\def\ugrbcon#1{\grbcon\relax{#1}}

\def\un{\mathscr{R}}

\def\tr{\mathscr{T}}

\def\ev{\mathrm{ev}}

\def\tocsection#1{\section*{#1}\addcontentsline{toc}{section}{#1}}

\def\mytitle{}
\def\zmptitle{
  \begin{tabular}{cc}
    \begin{minipage}[c]{0.4\textwidth}
      \begin{flushleft}
        \includegraphics[width=110pt]{../../tex/zmp}
      \end{flushleft}  
    \end{minipage}&
    \begin{minipage}[c]{0.55\textwidth}
      \begin{flushright}
      {\small\sf\@preprint}
      \end{flushright}
    \end{minipage}
  \end{tabular}
  \vskip 2cm}

\def\maketitle{
  \setlength{\parskip}{\myparskip}
  \newpage
  \noindent
  \mytitle
  \begin{center}
    \LARGE\@ntitle\\
    \if!\@subtitle!\else \smallskip\LARGE\@subtitle\\\fi
    \bigskip
    \if!\@author!\else   \bigskip\large\@author\\\fi
    \if!\@adress!\else   \bigskip\normalsize\@adress\\\fi
    \if!\@email!\else    \bigskip\normalsize\textit{\@email}\\\fi
  \end{center}
  \vskip 2cm\thispagestyle{empty}}

\def\kobiburl#1{
   \IfBeginWith
     {#1}
     {http://arxiv.org/abs/}
     {
        \kobibarxiv{#1}
     }
     {\kobiblink{#1}}}
\def\kobibarxiv#1{\href{#1}{\texttt{[arxiv:\StrGobbleLeft{#1}{21}]}}}
\def\kobiblink#1{Available as: \href{#1}{\texttt{#1}}}

\def\etalchar{$^{+}$}

\def\kobib#1{
  \begin{raggedright}

  \end{raggedright}}

\def\showcomments{ -- Comments suppressed}

\newif\if@fewtab\@fewtabtrue{
  \count255=\time\divide\count255 by 60
  \xdef\hourmin{\number\count255}
  \multiply\count255 by-60\advance\count255 by\time
  \xdef\hourmin{\hourmin:\ifnum\count255<10 0\fi\the\count255}}
\def\ps@draft{
  \let\@mkboth\@gobbletwo
  \def\@oddfoot{
    \hbox to 7 cm{\tiny \versionno\hfil}
    \hskip -7cm\hfil\rm\thepage\hfil{\tiny\draftdate}}
  \def\@oddhead{}
  \def\@evenhead{}
  \let\@evenfoot\@oddfoot}
\def\draftdate{\number\month/\number\day/\number\year\ \ \ \hourmin }
\newcommand\version[1]{
  \typeout{}\typeout{#1}\typeout{}
  \vskip-1.7cm \centerline{\fbox{{\normalsize\tt DRAFT -- #1 -- 
  \draftdate\showcomments}}} \vskip0.92cm}
\def\draft#1{
  \def\versionno{#1}
  \pagestyle{draft}\thispagestyle{draft}
  \gdef\@ntitle{\version\versionno \@title}
  \global\def\draftcontrol{1}}
\global\def\draftcontrol{0}

\makeatother

\newcommand{\alxydim}[2]{\begin{aligned}\xymatrix#1{#2}\end{aligned}}
\renewcommand{\to}{\!\xymatrix@C=0.5cm{\ar[r] &}}
\renewcommand{\mapsto}{\!\xymatrix@C=0.5cm{\ar@{|->}[r] &}\!}
\renewcommand{\Rightarrow}{\!\xymatrix@C=0.5cm{\ar@{=>}[r] &}\!}
\newcommand{\incl}{\!\xymatrix@C=0.5cm{\ar@{^(->}[r] &}\!}

\renewcommand\Leftrightarrow{\!\xymatrix@C=0.5cm{\ar@{<=>}[r] &}\!}

\usepackage[latin1]{inputenc}
\usepackage[english]{babel}

\def\quot#1{``#1''}

\title{String Connections and Chern-Simons Theory}
\author{Konrad Waldorf}
\adress{Department of Mathematics\\University of California, Berkeley\\970 Evans Hall \#3840\\Berkeley, CA 94720, USA}
\email{waldorf@math.berkeley.edu}

\kohyp
\usepackage{amstext}
\usepackage{amssymb}

\def\nablaa{\nabla_{\!\!A}}
\def\hc#1#2{\mathrm{h}_{#1}#2}

\begin{document}


\maketitle

\begin{abstract}
We present a finite-dimensional and smooth formulation of  string structures on  spin bundles. It uses trivializations of the Chern-Simons 2-gerbe  associated to this bundle. Our formulation is particularly suitable to deal with string connections: it enables us to prove that every string structure admits a  string connection and that the possible choices form an affine space. Further we discover  a new  relation between string connections, 3-forms  on the base manifold, and degree three differential cohomology. We also  discuss in detail the relation between our formulation of string connections and the original version of Stolz and Teichner.  
\showkeywords\begin{flushleft}\footnotesize\textbf{MSC classification}: primary 53C08, secondary 57R56, 57R15.\end{flushleft}
\end{abstract}

\tableofcontents

\tocsection{Introduction}

This article is concerned with a smooth manifold $M$, a principal spin bundle $P$ over $M$, and an additional structure on $P$ called \emph{string structure}. From a purely topological point of view, a string structure on $P$ is a certain  cohomology class $\xi\in\h^3(P,\Z)$.

We show that a string structure on  $P$ can equivalently  be understood as  a trivialization of a geometrical object over $M$, a bundle 2-gerbe $\mathbb{CS}_P$. This bundle 2-gerbe comes associated with $P$ and plays an important role in classical Chern-Simons theory. Our aim is to use this new formulation of string structures as a  suitable setup to study \emph{string connections}.

A string connection is an additional structure for a connection $A$ on $P$ and a string structure on $P$. The combination of a string structure and a string connection is  called a \emph{geometric string structure} on $(P,A)$. We give a new definition of a string connection in terms of connections on bundle 2-gerbes, and discuss in which sense it is equivalent to the original concept introduced by Stolz and Teichner \cite{stolz1}. Then, we provide the following list of new results on string connections:
\begin{itemize}
\item 
Geometric string structures form a 2-category. The set of isomorphism classes of geometric string structures is parameterized by  degree three differential cohomology of $M$. 
\item
Every string connection defines a  3-form on $M$, whose pullback to $P$ differs from the Chern-Simons 3-form associated to $A$ by a closed 3-form with integral periods.

\item
For every string structure on $P$ and every connection $A$ there exists a string connection. The set of possible string connections is an affine space. 
\end{itemize}

Back in the eighties, string structures emerged from  two-dimensional supersymmetric  field theories in work by Killingback \cite{killingback1} and Witten \cite{witten2}, and several definitions have been proposed so far. Firstly, McLaughlin \cite{mclaughlin1} defines a string structure on $P$ as a lift of the structure group of the \quot{looped} bundle $LP$ over $LM$ from $L\spin n$ to its universal central extension. 
Secondly,  Stolz and Teichner define a string structure on $P$ as a lift of the structure group of $P$ from $\spin n$ to a certain three-connected  extension, the  \emph{string group}. This group is \emph{not} a Lie group, but can be realized either as an infinite-dimensional, strict Fréchet Lie 2-group \cite{baez9}, or as a finite-dimensional, non-strict Lie 2-group \cite{pries2}. Thirdly, a string structure can be understood as a class $\xi\in\h^3(P,\Z)$, as mentioned in the beginning. This approach is  equivalent to the Stolz-Teichner definition, and one can show that string structures in that sense exist if and only if a certain class $\frac{1}{2}p_1(P) \in \h^4(M,\Z)$ vanishes.

Our formulation of string structures is based on the idea of Brylinski and McLaughlin to realize the obstruction class $\frac{1}{2}p_1(P)$ in a geometrical way \cite{brylinski7,brylinski4,brylinski5}. The bundle 2-gerbe $\mathbb{CS}_P$ we use here is such a geometrical realization, and goes back to work of Carey, Johnson, Murray and Stevenson \cite{carey4}. In this paper we go one step further: by considering string structures as  \emph{trivializations} of the bundle 2-gerbe $\mathbb{CS}_P$, we  also geometrically realize \emph{in which way} the obstruction may vanish.

The main innovation of our new formulation of a string structure is that it remains in the category of finite-dimensional, smooth manifolds, in contrast to the formulations of McLaughlin and Stolz-Teichner. This becomes possible by feeding    the \emph{basic gerbe} of Gaw\c edzki-Reis \cite{gawedzki1} and Meinrenken \cite{meinrenken1} into the construction of the bundle 2-gerbe $\mathbb{CS}_P$.  A second novelty is that our formulation exhibits string structures as objects in a 2-category rather than in a category: we thus uncover a further layer of structure.
The more recent formalism of Schommer-Pries \cite{pries2} has the same advantages, but has so far not been extended to string connections.

String \emph{connections} on top of string structures have been introduced by Stolz and Teichner as \quot{trivializations} of a certain version of classical Chern-Simons theory \cite{stolz1}. The definition we give here is again purely finite-dimensional and smooth:
it can be expressed as a couple of real-valued differential forms of degrees one and two.
Our concept can  systematically be understood as an application of the theory of connections on bundle 2-gerbes. In the proofs of our results we use and extend this existing theory.

Our results on string connections make progress in three aspects. Firstly, the action of degree three differential cohomology on geometric string structures that we have found generalizes the well-known action of ordinary degree three cohomology on string structures to  a setup \emph{with} connections. Secondly, we provide a new explanation for the appearance of 3-forms on $M$ in the presence of string structures. Such 3-forms have appeared before in work of Redden about string manifolds \cite{redden1}. We have now identified 3-forms as a genuine feature of a string connection. Thirdly, our results about the  existence of string connections and the structure  of the  space of string connections verify a conjecture of Stolz and Teichner \cite{stolz1}.

One thing one could further investigate  is the role of string connections in two-dimensional supersymmetric non-linear sigma models. In such models, string connections are supposed to provide for well-defined Feynman amplitudes \cite{alvarez3,freed5}. Further, sigma model actions can contain  Wess-Zumino terms; these are best understood in terms of gerbes with connection \cite{gawedzki3}. It would be interesting to better understand the interplay between string connections and Wess-Zumino terms, given  that both structures are now available in a formulation with gerbes.

This article is organized in the following way. In the next section we present a detailed overview on our results. Sections \ref{2gerbes} and \ref{con2gerbes} provide  full definitions and statements, and sketch the proofs. Section \ref{trivcs} is devoted to the comparison of our concepts with those of Stolz and Teichner. In Section \ref{background} we provide some background in bundle gerbe theory. Remaining  technical lemmata are collected in Section \ref{tech}.

\paragraph{Acknowledgements.} I gratefully acknowledge support by a Feodor-Lynen scholarship, granted by the Alex\-an\-der von Hum\-boldt Foundation. I thank Martin Olbermann, Arturo Prat-Waldron, Urs Schreiber and Peter Teichner for  exciting discussions, and two referees for their helpful comments and suggestions.
\setsecnumdepth{2}

\section{Summary of the Results}

We explain in Section \ref{sumstring} how string structures can be understood as trivializations of the Chern-Simons 2-gerbe. In Section \ref{sumstringcon} we upgrade to a setup with connections, and explain the relation between string connections and Chern-Simons theory. In Section \ref{resstringcon} we present our results on string connections.

\subsection{String Structures as Trivializations}

\label{sumstring}

Let us start with reviewing the  notion of a string structure on the basis of Stolz and Teichner \cite{stolz1} and Redden \cite{redden1}. Throughout this article $M$ is a smooth manifold and $n$ is an integer, $n=3$ or $n>4$. Note that  $\spin n$ is simple, connected, simply-connected and  compact. In particular, there is a standard generator of $\h^3(\spin n,\Z) \cong \Z$. 

Stolz and Teichner \cite{stolz1} define a string structure on a principal $\spin n$-bundle $P$ as a lift of the structure group of $P$ from $\spin n$ to a certain  extension
\begin{equation*}
\alxydim{}{B\ueins \ar[r] & \str n \ar[r] & \spin n\text{,}}
\end{equation*}
the \emph{string group}.
In this article we are not going to use the string group directly, since it is not a finite-dimensional Lie group, and our aim is to work in the finite-dimensional setting. Instead, we use the following simple replacement:
\begin{definition}[{\cite[Definition 6.4.2]{redden1}}]
\label{stringstructure}
Let $P$ be a principal $\spin n$-bundle over $M$. A \emph{string class} on $P$ is a class $\xi\in \h^3(P,\Z)$, such that for every point $p\in P$  the associated inclusion 
\begin{equation*} 
\iota_p: \spin n \to P: g \mapsto p.g
\end{equation*}
pulls $\xi$ back to the  standard generator of $\h^3(\spin n,\Z)$. 
\end{definition}

The string classes of  Definition \ref{stringstructure} are in bijection to equivalence classes of Stolz-Teichner string structures \cite[Proposition 6.4.3]{redden1}.
Basically, this bijection is obtained by considering a lift of the structure group of $P$ as a $B\ueins$-bundle over $P$. Such bundles have  Dixmier-Douady classes $\xi\in\h^3(P,\Z)$, and properties of  $\str n$ imply that these are string classes.

In the following, we give a new definition of a string structure that is equivalent to Definition \ref{stringstructure}, and hence  also equivalent to the definition of Stolz and Teichner. Before, we recall an important classification result for string structures. One can ask whether a principal $\spin n$-bundle $P$ admits string structures, and if it does, how many.
Associated to $P$ is a class in $\h^4(M,\Z)$ which is, when multiplied by two, the first Pontryagin class of the underlying $\so n$-bundle. Therefore it is denoted by $\frac{1}{2}p_1(P)$. In terms of string classes, the classification result is as follows.

\begin{theorem}[{\cite[Section 5]{stolz1}}]
\label{stringclassi}
Let $\pi:P \to M$ be a principal $\spin n$-bundle over $M$.
\begin{enumerate}
\item[(a)]
$P$ admits  string classes if and only if $\frac{1}{2}p_1(P)=0$. 

\item[(b)]
If $P$ admits string classes, the possible choices form a torsor over the group $\h^3(M,\Z)$, where the action of $\eta\in \h^3(M,\Z)$ takes a string class $\xi$ to the string class $\xi + \pi^{*}\eta$. 
\end{enumerate}
\end{theorem}

Next we report the first results of this article. It is well-known that the group $\h^4(M,\Z)$ classifies geometrical objects over $M$ called \emph{bundle 2-gerbes} \cite{stevenson2}. In this article, bundle 2-gerbes are defined internally to the category of smooth, finite-dimensional  manifolds (see Definition \ref{twogerbe}). The before-mentioned classification result for bundle 2-gerbes implies that there exists an isomorphism class of bundle 2-gerbes over $M$ with  characteristic class equal to $\frac{1}{2}p_1(P)$. Our first result is the observation that this isomorphism class has a distinguished representative: the Chern-Simons 2-gerbe $\mathbb{CS}_P$.

\begin{theorem}
\label{csexist}
Let $P$ be a principal $\spin n$-bundle over $M$. Then, the Chern-Simons 2-gerbe $\mathbb{CS}_P$ has  characteristic class  $\frac{1}{2}p_1(M)$. \end{theorem} 

We explain the construction of the Chern-Simons 2-gerbe $\mathbb{CS}_P$ in Section \ref{cstwogerbe} following Carey et al. \cite{carey4}: its first step is to  pull back the \emph{basic gerbe} $\mathcal{G}$ over $\spin n$ along the map $P \times_M P \to \spin n$. The characteristic class of $\mathbb{CS}_P$ is calculated by combining a result of McLaughlin (\cite[Lemma 2.2]{mclaughlin1}, recalled below as Lemma \ref{cccs}) with a general result about Chern-Simons 2-gerbes (Lemma \ref{lem:mcl}). A different method carried out in \cite[Proposition 9.3]{stevenson2} (with an infinite-dimensional version of the Chern-Simons 2-gerbe) is to extract an explicit \v Cech 4-cocycle, and noticing that it is the same as found by Brylinski and McLaughlin in the world of sheaves of 2-groupoids \cite{brylinski7}. 

In order to detect whether or not $\frac{1}{2}p_1(P)$ vanishes, we   use a geometrical criterion for the vanishing of the characteristic class of a bundle 2-gerbe, called  \emph{trivialization}. In Section \ref{trivstring} we will see that trivializations of a bundle 2-gerbe $\mathbb{G}$ form a 2-groupoid $\triv{\mathbb{G}}$ (Lemma \ref{trivgroupoid}). Here we just mention that a trivialization $\mathbb{T}$ of $\mathbb{CS}_P$ determines a class $\xi_{\mathbb{T}} \in \h^3(P,\Z)$; see Proposition \ref{stringstructureaustriv}. Our second result is the following theorem.

\begin{theorem}
\label{class}
The bundle $P$ admits string classes if and only if the Chern-Simons 2-gerbe $\mathbb{CS}_P$ has a trivialization. In that case, the assignment $\mathbb{T}\mapsto \xi_{\mathbb{T}}$ establishes a bijection
\begin{equation*}
\bigset{4.2cm}{Isomorphism classes of trivializations of $\mathbb{CS}_P$} \cong \bigset{3.5cm}{String classes on $P$}\text{.}
\end{equation*}
\end{theorem}

The first part follows directly from a general result of Stevenson \cite{stevenson2}, see Lemma \ref{trivvanish}. The second part is Proposition \ref{stringbijection} in Section \ref{trivstring}: there we observe that both sides are torsors over the group $\h^3(M,\Z)$, and prove that $\mathbb{T}\mapsto \xi_{\mathbb{T}}$ is equivariant.

Theorem \ref{class} motivates the following new definition of a string structure:

\begin{definition}
\label{stringstructurenew}
Let $P$ be a principal $\spin n$-bundle over $M$. A \emph{string structure} on $P$ is a trivialization of the Chern-Simons 2-gerbe $\mathbb{CS}_P$.
\end{definition}

As a direct consequence of this definition, string structures are not just a set (like string classes), or a category (as in \cite{stolz1}):

\begin{corollary}
\label{string2cat}
String structures on a principal $\spin n$-bundle $P$ over $M$  in the sense of Definition \ref{stringstructurenew} form a 2-groupoid $\triv {\mathbb{CS}_P}$.
\end{corollary}

We remark that the notions of string structures introduced later by Schommer-Pries \cite{pries2} and Fiorenza, Schreiber and Stasheff \cite{Fiorenza} have the same feature.

\subsection{String Connections as Connections on Trivializations}

\label{sumstringcon}

The main objective of this article is to show that Definition \ref{stringstructurenew} --- understanding a string structure as a trivialization of the Chern-Simons 2-gerbe --- has many advantages, in particular when additional, differential-geometric structures become involved.

An example for such a differential-geometric structure is a \emph{string connection}. Our first goal is to give a  thorough definition of a string connection. We will use existing concepts and results on connections on bundle 2-gerbes. Basically, a connection on a bundle 2-gerbe is the same type of additional structure an ordinary connection is for an ordinary principal bundle. For example, every connection on a bundle 2-gerbe has a curvature, which   is a closed 4-form on $M$ that represents the characteristic class of the underlying bundle 2-gerbe in the real cohomology of $M$.  

\begin{theorem}
\label{cancon}
Any connection $A$ on a principal $\spin n$-bundle $P$ over $M$ determines a connection $\nablaa$ on the Chern-Simons 2-gerbe $\mathbb{CS}_P$. The  curvature of $\nablaa$ is one-half of the Pontryagin 4-form associated to $A$. 
\end{theorem}

In Section \ref{sec:conncs} we give an explicit construction of the connection $\nablaa$ based on results about connections on multiplicative bundle gerbes \cite{carey4,waldorf5}. It involves the Chern-Simons 3-form $TP(A) \in \Omega^3(P)$ which is fundamental in Chern-Simons theory  \cite{chern1}; hence the name of the 2-gerbe $\mathbb{CS}_P$. The curvature of $\nablaa$ is calculated in Lemma \ref{cscurv}.

Now that we have equipped the Chern-Simons 2-gerbe $\mathbb{CS}_P$  with the connection $\nablaa$,  we  consider trivializations of $\mathbb{CS}_P$ that preserve the connection $\nablaa$ in an appropriate way. However, the term \quot{preserve} is not accurate: for a morphism between two bundle 2-gerbes  with connections, being \quot{connection-preserving} is \emph{structure}, not  property.
So we better speak of \emph{trivializations with compatible connection}. Generally, if $\mathbb{G}$ is a bundle 2-gerbe with connection $\nabla$, there is a 2-groupoid $\triv {\mathbb{G},\nabla}$ of trivializations of $G$ with connection compatible with $\nabla$. 

Accompanying Definition \ref{stringstructurenew} we propose the following definition. 
\begin{definition}
\label{stringcon}
Let $P$ be a principal $\spin n$-bundle over $M$ with connection $A$, and let $\mathbb{T}$ be a string structure on $P$. A \emph{string connection} for $(\mathbb{T},A)$ is a connection $\babla$ on $\mathbb{T}$  that is compatible with $\nablaa$. A \emph{geometric string structure} on $(P,A)$ is a pair $(\mathbb{T},\babla)$ of a string structure $\mathbb{T}$ on $P$ and a string connection $\babla$ for $(\mathbb{T},A)$.
\end{definition} 

We suppress the details behind this definition until Section \ref{contriv}. We only remark that every string connection $\babla$ determines  a closed  3-form $K_{\babla} \in \Omega_{\mathrm{cl}}^3(P)$ that represents  the string class $\xi_{\mathbb{T}}\in \h^3(P,\Z)$ corresponding to  $\mathbb{T}$ under the bijection of Theorem \ref{class}. This 3-form appears below in Theorem \ref{threeform}.

In the remainder of this subsection we  compare  Definition \ref{stringcon} with the original concept of a geometric string structure introduced by Stolz and Teichner \cite{stolz1}.
There, a \emph{geometric string structure} on a principal $\spin n$-bundle $P$ with connection $A$ is by definition a \emph{trivialization of the extended Chern-Simons theory} $Z_{P,A}$ associated to $(P,A)$ \cite[Definition 5.3.4]{stolz1}. Chern-Simons theory is viewed as a three-dimensional extended topological field theory: it assigns values $Z_{P,A}(\phi^d)$ to smooth maps $\phi^d: X^d \to M$ where $X^d$ is a $d$-dimensional  oriented manifold and $d$ varies from $0$ to $3$. These  values are supposed to be compatible with the gluing of manifolds along boundaries, and are supposed to depend in a smooth way on the map $\phi^d$. In the top dimension, $Z_{P,A}(\phi^{3})$ is an element of $\ueins$, namely the ordinary classical Chern-Simons invariant of $X^3$. According to Stolz and Teichner, a \emph{trivialization of $Z_{P,A}$} is a similar assignment that \quot{lifts} the values of $Z_{P,A}$ in a certain way; for example, it assigns to $\phi^3$ a real number whose exponential is  $Z_{P,A}(\phi^{3})$.

Our argument uses that the extended Chern-Simons theory $Z_{P,A}$ can directly be formulated in terms of the Chern-Simons 2-gerbe $\mathbb{CS}_P$ and its connection $\nablaa$ (Definition \ref{def:cs}). This formulation is close to  Freed's original definition of extended Chern-Simons theory by \emph{transgression} \cite{freed3}. The difference to the formulation  used by Stolz and Teichner is essentially a different choice of a geometrical model for $K(\Z,4)$: Stolz and Teichner use the space of torsors over the outer automorphisms of a certain von Neumann algebra, while the model that stands behind bundle 2-gerbes is the topological group $BBB\ueins$. 

In this article we  use the formulation of the Chern-Simons theory $Z_{P,A}$ in terms of the Chern-Simons 2-gerbe $\mathbb{CS}_P$, and adopt the Stolz-Teichner concept of a trivialization to this formulation (Definition \ref{def:trivcs}). Now, our  concept of a geometric string structure  matches almost literally the definition of Stolz and Teichner: it trivializes the Chern-Simons 2-gerbe $\mathbb{CS}_P$ and the connection $\nablaa$ and thereby the Chern-Simons theory $Z_{P,A}$. More precisely, in Section \ref{trivcsdet} we construct a map
\begin{equation*}
\mathscr{S}:\bigset{3.7cm}{Isomorphism classes of geometric string structures on $(P,A)$} \to
\bigset{5.4cm}{Isomorphism classes of trivializations of the extended Chern-Simons theory $Z_{P,A}$}
\end{equation*}
where the geometric string structures on the left hand side are those of Definition \ref{stringcon}, and the trivializations on the right hand side are the ones  of Stolz and Teichner.

\begin{theorem}
\label{stringtrivcs}
Let $P$ be a principal $\spin n$-bundle over  $M$ with connection $A$, and let $Z_{P,A}$ be the extended Chern-Simons theory associated to $(P,A)$. Then, the map
$\mathscr{S}$ is injective. Moreover, under the assumption that the cobordism hypothesis holds for $Z_{P,A}$, it is also surjective. 
\end{theorem}
 
In that sense, our concept of a geometric string structure coincides with the one of Stolz and Teichner. 
We prove Theorem \ref{stringtrivcs} in Section \ref{trivcs} as a combination of Lemmata \ref{lem:inj} and \ref{lem:surj}. As we will see, the injectivity of $\mathscr{S}$ relies on the \emph{smoothness} of the assignments of $Z_{P,A}$.  Surjectivity holds if we assume that both $Z_{P,A}$ and its trivializations are already determined by their value on a point $\phi^0:X^0 \to M$, which is essentially the statement of the cobordism hypothesis recently proved by Lurie \cite{lurie1}. Whether or not that assumption is true remains undecidable because some aspects of the field theory $Z_{P,A}$  are not treated sufficiently in \cite{stolz1} --- most importantly their functorality with respect to manifolds with corners.

\subsection{Results on String Connections}

\label{resstringcon}

We have argued that trivializations of the Chern-Simons 2-gerbe $\mathbb{CS}_P$ with compatible connection are an appropriate formulation for geometric string structures on a pair $(P,A)$ of a principal $\spin n$-bundle $P$ over $M$ and a connection $A$ on $P$. In this formulation, we prove the following results on string connections.

Two results are direct consequences of the general theory of  bundle 2-gerbes with connection, stated as Lemma \ref{catgeomstr} below. The first is an analog of Corollary \ref{string2cat}.

\begin{corollary}
\label{geomstrcat}
Geometric string structures on a principal $\spin n$-bundle $P$ over $M$ with connection $A$ form a 2-groupoid $\triv{\mathbb{CS}_P,\nablaa}$.
\end{corollary}

We remark that the formalism of \quot{string bundles with connection} introduced later in \cite{Fiorenza} has the same feature. 
The second result  extends  Theorem \ref{stringclassi} from string structures to geometric string structures.

\begin{corollary}
\label{geomstrtorsor}
Let $P$ be a principal $\spin n$-bundle over $M$ with connection $A$, and suppose $\frac{1}{2}p_1(P)=0$. Then,
$(P,A)$ admits a geometric string structure. The set of isomorphism classes of geometric string structures is a torsor over  the differential cohomology group $\hat\h^3(M,\Z)$.
\end{corollary}

Here we understand differential cohomology in the axiomatic sense of Simons and Sullivan \cite{simons1}. In the context of bundle gerbes and bundle 2-gerbes, Deligne cohomology is often an appropriate realization. We recall that one  feature of differential cohomology is a commutative diagram:
\begin{equation*}
\alxydim{}{\hat\h^k(M,\Z) \ar[r]^-{\mathrm{pr}} \ar[d]_{\Omega} & \h^k(M,\Z) \ar[d]^{\iota^{*}} \\ \Omega_{\mathrm{cl}}^k(M) \ar[r] & \h^k(M,\R)\text{.}}
\end{equation*}
Acting with a class $\kappa\in\hat\h^3(M,\Z)$ on a geometric string structure covers the  action of $\mathrm{pr}(\kappa)\in\h^3(M,\Z)$ on the underlying string structure from Theorem \ref{stringclassi} (b).

The next result is an interesting relation between geometric string structures and 3-forms on $M$. 
\begin{theorem}
\label{threeform}
Let $\pi:P \to M$ be a principal $\spin n$-bundle over $M$ with a connection $A$. 
Let $(\mathbb{T},\babla)$ be a geometric string structure on $(P,A)$. Then, there exists a unique 3-form $H_{\babla}\in\Omega^3(M)$ such that
\begin{equation}
\label{threeformdiff}
\pi^{*}H_{\babla} = K_{\babla}+TP(A)\text{,}
\end{equation}
where $K_{\babla}$ is the   3-form that represents the string class $\xi_{\mathbb{T}} \in \h^3(P,\Z)$, and $TP(A)$ is the Chern-Simons 3-form associated to the connection $A$.
Moreover, $H_{\babla}$ has the following properties:
\begin{enumerate}
\item[(a)]
Its derivative $\mathrm{d}H_{\babla}$ is one-half of the Pontryagin 4-form of $A$.

\item[(b)] 
It depends only on the isomorphism class of $(\mathbb{T},\babla)$.
\item[(c)]
For $\kappa\in\hat \h^3(M,\Z)$ we have
\begin{equation*}
H_{\babla.\kappa} = H_{\babla} + \Omega(\kappa)\text{.}
\end{equation*}
under the action of Corollary \ref{geomstrtorsor}.
\end{enumerate}
\end{theorem}

This theorem follows from general results about connections on bundle 2-gerbes that we derive in Lemmata \ref{3form} and \ref{3form2}. 

Redden shows with analytical methods \cite{redden1} in the case that $P$ is a spin structure on a Riemannian manifold $(M,g)$ and $A$ is the Levi-Cevita connection, a string class $\xi$ determines a 3-form $H_{\xi}$ on $M$ with properties analogous to \erf{threeformdiff} and (a) of Theorem \ref{threeform}. One can  show that for every string class $\xi$ there exists a string connection $\babla$ such that $H_{\xi}= H_{\babla}$ \cite{redden2}, but it is an open problem whether or not such a string connection is determined by the given data.

Our last result is
the following theorem.

\begin{theorem}
\label{stringconcon}
For $P$ a principal $\spin n$-bundle over $M$, let $\mathbb{T}$ be a string structure. For any choice of a connection $A$ on $P$, there exists a string connection $\babla$ on $\mathbb{T}$. The set of  string connections on $\mathbb{T}$ is an affine space.
\end{theorem}

This result verifies a conjecture remarked by Stolz and Teichner \cite[Theorem 5.3.5]{stolz1}, when this conjecture is interpreted in our framework. For the proof of Theorem \ref{stringconcon}, we show in full generality that every trivialization of every bundle 2-gerbe with connection admits a connection (Proposition \ref{trivconex}), and that the set of such connections is an affine space (Proposition \ref{trivconaff}).

\section{Bundle 2-Gerbes and their Trivializations}

\label{2gerbes}

We provide  definitions of bundle 2-gerbes and  trivializations, describe their relation to string structures and prove the results stated in Section \ref{sumstring}.
Connections will be added to the picture in Section \ref{con2gerbes}.
\subsection{The Chern-Simons Bundle 2-Gerbe}

\label{cstwogerbe}
 
We give the details for the proof of Theorem \ref{csexist}: we describe the construction of the Chern-Simons 2-gerbe $\mathbb{CS}_P$ associated to a principal $\spin n$-bundle $P$.

First we recall the definition of a bundle 2-gerbe, which is based very much on the notion a bundle (1-)gerbe introduced by Murray \cite{murray}. All bundle gerbes in this article have structure group $\ueins$.   Since bundle gerbes become more and more common, I only recall a few facts at this place; additional background is provided in Section \ref{background}. Apart from these, the reader is referred to \cite{murray3,schweigert2} for introductions, and to \cite{stevenson1,waldorf1} for detailed treatments. 
\begin{enumerate}
\item 
For $M$ a smooth manifold, bundle gerbes over $M$ form a strictly monoidal 2-groupoid $\ugrb M$. Its  1-morphisms will be called \emph{isomorphisms}, and its 2-morphisms will be called \emph{transformations}.

\item
Denoting by $\hc 0 \ugrb M$ the group of isomorphism classes of bundle gerbes over $M$, there is a  group isomorphism
\begin{equation*}
\mathrm{DD}: \hc 0 \ugrb M \to \h^3(M,\Z)\text{.}
\end{equation*}
For $\mathcal{G}$ a bundle gerbe, $\mathrm{DD}(\mathcal{G})$ is called the \emph{Dixmier-Douady class} of $\mathcal{G}$.

\item
For $f:M \to N$ a smooth map, there is a pullback 2-functor
\begin{equation*}
f^{*}: \ugrb N \to \ugrb M\text{,}
\end{equation*}
and whenever smooth maps are composable, the associated 2-functors compose strictly. 
\end{enumerate}

Let us fix some notation. We say that a \emph{covering} of a smooth manifold $M$ is a surjective submersion $\pi:Y \to M$. We denote the $k$-fold fibre product of $Y$ with itself by $Y^{[k]}$. This is again a smooth manifold, and for integers $k>p$, all the projections 
\begin{equation*}
\pi_{i_1,...,i_p}:Y^{[k]} \to Y^{[p]}:(y_1,...,y_k) \mapsto (y_{i_1},...,y_{i_p})
\end{equation*}
are smooth maps.

\begin{definition}[{\cite[Definition 5.3]{stevenson2}}]
\label{twogerbe}
A \emph{bundle 2-gerbe} over $M$ is a covering $\pi:Y \to M$ together with a bundle gerbe $\mathcal{P}$ over $Y^{[2]}$, an isomorphism 
\begin{equation*}
\mathcal{M}: \pi_{12}^{*}\mathcal{P} \otimes \pi_{23}^{*}\mathcal{P} \to \pi_{13}^{*}\mathcal{P}
\end{equation*}
of bundle gerbes over $Y^{[3]}$, and a transformation 
\begin{equation*}
\alxydim{@=1.8cm}{\pi_{12}^{*}\mathcal{P} \otimes \pi_{23}^{*}\mathcal{P} \otimes \pi_{34}^{*}\mathcal{P} \ar[r]^-{\pi_{123}^{*}\mathcal{M} \otimes \id} \ar[d]_{\id \otimes \pi_{234}^{*}\mathcal{M}} & \pi_{13}^{*}\mathcal{P} \otimes \pi_{34}^{*}\mathcal{P} \ar@{=>}[dl]|*+{\mu} \ar[d]^{\pi_{134}^{*}\mathcal{M}} \\ \pi_{12}^{*}\mathcal{P} \otimes \pi_{24}^{*}\mathcal{P} \ar[r]_-{\pi_{124}^{*}\mathcal{M}} & \pi_{14}^{*}\mathcal{P}}
\end{equation*}
over $Y^{[4]}$ that satisfies the  pentagon axiom shown in Figure \ref{fig:pentagon}. \end{definition}

The isomorphism $\mathcal{M}$ is called \emph{product} and the transformation $\mu$ is called the \emph{associator}. The pentagon axiom implies the cocycle condition for a certain degree three cocycle on $M$ with values in $\ueins$, which defines --- via the exponential sequence --- a class 
\begin{equation*}
\mathrm{CC}(\mathbb{G}) \in \h^4(M,\Z)\text{;}
\end{equation*}
see \cite[Proposition 7.2]{stevenson2} for the details. 
In the following we call $\mathrm{CC}(\mathbb{G})$ the \emph{characteristic class} of the bundle 2-gerbe $\mathbb{G}$.

\begin{figure}[t]
\begin{equation*}
\alxydim{@C=0.5cm@R=1.2cm}{&&\ast \ar@{=>}[dll]_*+{\id \circ (\pi_{1234}^{*}\mu \otimes \id)} \ar@{=>}[drr]^*+{\pi_{1345}^{*}\mu \circ \id}&&\\\ast \ar@{=>}[dr]_*+{\pi_{1245}^{*}\mu \circ \id}&&&&\ast \ar@{=>}[dl]^*+{\pi_{1235}^{*}\mu \circ \id}\\&\ast \ar@{=>}[rr]_*+{\id \circ (\id \otimes \pi_{2345}^{*}\mu)}&&\ast&}
\end{equation*}
\caption{The pentagon axiom for a bundle gerbe product $\mu$. It is an equality between transformations over $Y^{[5]}$. }
\label{fig:pentagon}
\end{figure}

Now let $G$ be a  Lie group. In fact we are only interested in $G=\spin n$, but the following construction applies in general. For more details on this construction, we refer the reader to \cite{carey4}.

Let $\pi: P \to M$ be a principal $G$-bundle over $M$.
The key idea in the construction of the Chern-Simons 2-gerbe $\mathbb{CS}_P$ is to take the bundle projection $\pi:P \to M$ as its covering. Its two-fold fibre product comes with a  smooth map $g:P^{[2]} \to G$ defined by $p'.g(p,p')=p$, for $p,p'\in P$ two points in the same fibre. Suppose now we have a bundle gerbe $\mathcal{G}$  over $G$ available. Then we put
\begin{equation*}
\mathcal{P} := g^{*}\mathcal{G}
\end{equation*}
as the bundle gerbe of  $\mathbb{CS}_P$.
 It turns out that the remaining structure, namely the product $\mathcal{M}$ and the associator $\mu$, can be induced from additional structure on the bundle gerbe $\mathcal{G}$ over $G$, called a \emph{multiplicative structure}. Bundle gerbes with multiplicative structure are called \emph{multiplicative bundle gerbes}. 

We have thus a bundle 2-gerbe $\mathbb{CS}_P(\mathcal{G})$ associated to every principal $G$-bundle $P$ and every multiplicative bundle gerbe $\mathcal{G}$ over $G$.
We remark that multiplicative bundle gerbes $\mathcal{G}$ over $G$ are classified up to isomorphism by $\h^4(BG,\Z)$ via a \emph{multiplicative class} $\mathrm{MC}(\mathcal{G})\in \h^4(BG,\Z)$ \cite{carey4}. The transgression homomorphism 
\begin{equation*}
\h^4(BG,\Z) \to \h^3(G,\Z)
\end{equation*} 
takes the multiplicative class  to the Dixmier-Douady class of the underlying bundle gerbe; see \cite[Proposition 2.11]{waldorf5}.

The relation between the  multiplicative class of a multiplicative bundle gerbe $\mathcal{G}$ over $G$ and the characteristic class of the associated Chern-Simons 2-gerbe $\mathbb{CS}_P(\mathcal{G})$ is the following.

\begin{lemma}[{\cite[Theorem 3.13]{waldorf5}}]
\label{cccs}
Let $P$ be a principal $G$-bundle over $M$, and let $\eta_P: M \to BG$ be a classifying map for $P$. Then,
\begin{equation*}
\mathrm{CC}(\mathbb{CS}_P(\mathcal{G})) = \eta^{*}_P\mathrm{MC}(\mathcal{G})\text{.}
\end{equation*}
\end{lemma}

Let us restrict to $G=\spin n$, and  assume a principal $\spin n$-bundle $P$ over $M$ with a classifying map $\eta_P$. Suppose we have a multiplicative bundle gerbe $\mathcal{G}$ over $G$ such that
\begin{equation}
\label{2}
\eta_P^{*}\mathrm{MC}(\mathcal{G}) = {\textstyle\frac{1}{2}}p_1(P) \in \h^4(M,\Z)\text{.}
\end{equation}
Then, by Lemma \ref{cccs}, we obtain
\begin{equation*}
\mathrm{CC}(\mathbb{CS}_P(\mathcal{G})) = \textstyle{\frac{1}{2}}p_1(P)\text{,}
\end{equation*}
which proves Theorem \ref{csexist}.

In order to find a  multiplicative bundle gerbe $\mathcal{G}$ satisfying \erf{2} we use the following result of McLaughlin.

\begin{lemma}[{\cite[Lemma 2.2]{mclaughlin1}}]
\label{lem:mcl}
Let $P$ be a principal $\spin n$-bundle over $M$, and let $\eta_P:M \to B\spin n$ be a classifying map for $P$. Then, there is a unique $\tau\in\h^4(B\spin n,\Z)$  whose transgression is the standard generator of $\h^3(\spin n, \Z)$. Moreover,
\begin{equation*}
\eta_P^{*}\tau = \textstyle{\frac{1}{2}}p_1(P)\text{.}
\end{equation*}
\end{lemma}

The lemma tells us that the multiplicative bundle gerbe $\mathcal{G}$ we want to find has multiplicative class equal to $\tau$. In turn, this means that the underlying bundle gerbe $\mathcal{G}$ has a Dixmier-Douady class equal to the standard generator of $\h^3(\spin n,\Z)$. 
Such a bundle gerbe is well-known: the \emph{basic gerbe} over $\spin n$. The basic gerbe  has been constructed by Gaw\c edzki-Reis \cite{gawedzki1} and Meinrenken \cite{meinrenken1} in a smooth and finite-dimensional way using Lie-theoretical methods.

Over simple, connected, simply-connected Lie groups such as $\spin n$, every bundle gerbe carries a unique multiplicative structure, up to 1-isomorphism. In fact, Schommer-Pries has proved that the multiplicative structure is unique up to a \emph{contractible} choice of isomorphisms, i.e. each two 1-isomorphisms are related by a unique 2-isomorphism \cite[Theorem 99]{pries2}. 

However, for the \emph{basic gerbe} $\mathcal{G}$ it is  possible to select a \emph{specific} multiplicative structure, as explained in detail in \cite{Waldorf}. Briefly, we consider the (infinite-dimensional) bundle gerbe $\un(\tr_\mathcal{G})$ over $G$, obtained from $\mathcal{G}$ by transgression $\tr$, followed by regression $\un$; see \cite{waldorf10}. By \cite[Theorem A]{waldorf10}, the two bundle gerbes $\mathcal{G}$  and $\un(\tr_\mathcal{G})$ are  isomorphic via a canonical isomorphism $\mathcal{A}: \mathcal{G} \to \mathscr{R}(\mathscr{T}_{\mathcal{G}})$. 
The point is that the infinite-dimensional bundle gerbe $\un(\tr_\mathcal{G})$ carries a  multiplicative structure given by the  \emph{Mickelsson product} (see \cite{waldorf5} and \cite{mickelsson1}).  The isomorphism $\mathcal{A}$ pulls back this multiplicative structure to the basic gerbe $\mathcal{G}$. Since $\mathcal{G}$ is a finite-dimensional bundle gerbe, all its isomorphisms and transformations descend to finite-dimensional ones \cite[Theorem 1]{waldorf1}. That way, the pulled back infinite-dimensional multiplicative structure descends to a finite-dimensional one. 

Summarizing, the basic gerbe $\mathcal{G}$ over $\spin n$ comes equipped with a multiplicative structure. The associated Chern-Simons bundle 2-gerbe $\mathbb{CS}_P(\mathcal{G})$ is the one that is relevant in this article, and therefore abbreviated by $\mathbb{CS}_P$.

\begin{remark}
Multiplicative gerbes over a compact Lie group $G$  are the same as central Lie 2-group extensions of $G$  by $\mathcal{B}\ueins$ in the sense of Schommer-Pries \cite{pries2}; see \cite[Remark 101]{pries2} and \cite[Theorem 3.2.5]{Waldorf}. Under this equivalence, the infinite-dimensional multiplicative bundle gerbe $\un(\tr_\mathcal{G})$ corresponds to an infinite-dimensional model for the string group, similar to the one of \cite{baez9}, but still different. The finite-dimensional, multiplicative bundle gerbe $\mathcal{G}$ corresponds to a  finite-dimensional, smooth model for the string group.
\end{remark}

\subsection{Trivializations and String Structures}

\label{trivstring}

Next is the proof of  Theorem \ref{class}: we explain the relation between trivializations of the Chern-Simons 2-gerbe $\mathbb{CS}_P$ and string classes on $P$ in the sense of Definition \ref{stringstructure}.
 
\begin{definition}[{\cite[Definition 11.1]{stevenson2}}]
\label{deftriv}
Let $\mathbb{G}=(Y,\mathcal{P},\mathcal{M},\mu)$ be a bundle 2-gerbe over $M$. 
A \emph{trivialization} of $\mathbb{G}$ is a bundle gerbe $\mathcal{S}$ over $Y$, together with an isomorphism
\begin{equation*}
\mathcal{A}: \mathcal{P} \otimes \pi_2^{*}\mathcal{S} \to \pi_1^{*}\mathcal{S}
\end{equation*}
of bundle gerbes over $Y^{[2]}$ and a transformation
\begin{equation*}
\alxydim{@=1.3cm}{\pi_{12}^{*}\mathcal{P} \otimes \pi_{23}^{*}\mathcal{P} \otimes \pi_3^{*}\mathcal{S} \ar[r]^-{\id \otimes \pi_{23}^{*}\mathcal{A}} \ar[d]_{\mathcal{M} \otimes \id} & \pi_{12}^{*}\mathcal{P} \otimes \pi_{2}^{*}\mathcal{S} \ar@{=>}[dl]|*+{\sigma} \ar[d]^{\pi_{12}^{*}\mathcal{A}} \\ \pi_{13}^{*}\mathcal{P} \otimes \pi_{3}^{*}\mathcal{S} \ar[r]_-{\pi_{13}^{*}\mathcal{A}} & \pi_1^{*}\mathcal{S}}
\end{equation*}
over $Y^{[3]}$ that is compatible with the associator $\mu$ in the sense of Figure \ref{compass}.
\end{definition}

\begin{figure}[t]
\begin{footnotesize}
\begin{equation*}
\alxydim{@C=-6.2cm@R=1.6cm}{&&\pi_{14}^{*}\mathcal{A} \circ (\pi_{134}^{*}\mathcal{M} \otimes \id) \circ (\pi_{123}^{*}\mathcal{M} \otimes \id \otimes \id) \ar@/^2.3pc/@{=>}[drr]^>>>>>>>*+{\pi_{134}^{*}\sigma} \ar@/_2.3pc/@{=>}[dll]_>>>>>>>*+{\id \circ (\mu \otimes \id)}&&\\\pi_{14}^{*}\mathcal{A} \circ (\pi_{124}^{*}\mathcal{M} \otimes \id) \circ (\id \otimes \pi_{234}^{*}\mathcal{M} \otimes \id) \ar@/_0.8pc/@{=>}[dr]_*+{\pi_{124}^{*}\sigma \circ \id}&&&&\pi_{13}^{*}\mathcal{A} \circ (\pi_{123}^{*}\mathcal{M} \otimes \id) \circ (\id \otimes \id\otimes \pi_{34}^{*}\mathcal{A}) \ar@/^0.8pc/@{=>}[dl]^*+{\pi_{123}^{*}\sigma \circ \id}\\&\pi_{12}^{*}\mathcal{A} \circ (\id \otimes \pi_{24}^{*}\mathcal{A}) \circ (\id \otimes \pi_{234}^{*}\mathcal{M} \otimes \id)\hspace{1cm} \ar@{=>}@/_2.5pc/[rr]_*+{\id \circ (\id \otimes \pi_{234}^{*}\sigma)}&\hspace{11cm}& \hspace{1cm}\pi_{12}^{*}\mathcal{A} \circ (\id \otimes \pi_{23}^{*}\mathcal{A}) \circ (\id \otimes \id\otimes \pi_{34}^{*}\mathcal{A})&}
\end{equation*}
\end{footnotesize}
\caption{The compatibility condition between the associator $\mu$ of a bundle 2-gerbe and the transformation $\sigma$ of a trivialization. It is an equality of transformations over $Y^{[4]}$.}
\label{compass}
\end{figure}

The purpose of a trivialization is the following.

\begin{lemma}[{\cite[Proposition 11.2]{stevenson2}}]
\label{trivvanish}
The characteristic class of a bundle 2-gerbe vanishes if and only if it admits a trivialization. 
\end{lemma}

Applied to the Chern-Simons 2-gerbe $\mathbb{CS}_P$, Lemma \ref{trivvanish} proves the first part of Theorem \ref{class}: $P$ admits a string class  if and only if $\mathbb{CS}_P$ has a trivialization. The second part follows from Propositions \ref{stringstructureaustriv} and \ref{stringbijection} below. The first proposition explains how a trivialization defines a string class.

\begin{proposition}
\label{stringstructureaustriv}
Let $P$ be a principal $\spin n$-bundle over $M$, and let $\mathbb{T}=(\mathcal{S},\mathcal{A},\sigma)$ be a trivialization of the Chern-Simons 2-gerbe $\mathbb{CS}_P$. Then, 
\begin{equation*}
\xi_{\mathbb{T}} := \mathrm{DD}(\mathcal{S}) \in \h^3(P,\Z)
\end{equation*}
is a string class on $P$.
\end{proposition}

\begin{proof}
We have to show that the pullback of $\mathrm{DD}(\mathcal{S})$ along an inclusion $\iota_p:\spin n \to P$ is the generator of $\h^3(\spin n,\Z)$. By construction, this is nothing but the Dixmier-Douady class of the basic gerbe $\mathcal{G}$ which entered the structure of the Chern-Simons 2-gerbe $\mathbb{CS}_P$. 

Consider the isomorphism $\mathcal{A}:g^{*}\mathcal{G} \otimes \pi_2^{*}\mathcal{S} \to \pi_1^{*}\mathcal{S}$ of bundle gerbes over $P^{[2]}$ which is part of $\mathbb{T}$. For 
\begin{equation*}
s_p: \spin n \to P^{[2]}: g \mapsto (p.g,p)\text{,}
\end{equation*}
we observe that $g\circ s_p= \id$, $\pi_1 \circ s_p = \iota_p$, and $\pi_2 \circ s_p$ is a constant map. Thus, the pullback $s_p^{*}\mathcal{A}$ implies on the Dixmier-Douady classes
\begin{equation*}
\mathrm{DD}(\mathcal{G}) = \mathrm{DD}(\iota_p^{*}\mathcal{S})\text{,}  
\end{equation*}
since a bundle gerbe over a point has vanishing Dixmier-Douady class. 
\end{proof}

For the second proposition we need two lemmata.

\begin{lemma}
\label{trivgroupoid}
Trivializations of a fixed bundle 2-gerbe $\mathbb{G}$ form a 2-groupoid denoted $\triv{\mathbb{G}}$.
\end{lemma}

This also implies Corollary \ref{string2cat}. For the moment it is enough to note that a 1-morphism between trivializations $(\mathcal{S}_1,\mathcal{A}_1,\sigma_1)$ and $(\mathcal{S}_2,\mathcal{A}_2,\sigma_2)$ involves an isomorphism $\mathcal{B}:\mathcal{S}_1 \to \mathcal{S}_2$ between the two bundle gerbes.
We defer a complete treatment of the 2-groupoid $\triv {\mathbb{G}}$ to Section \ref{triv2cat}.

The next lemma equips the 2-groupoid $\triv{\mathbb{G}}$ with an additional structure.

\begin{lemma}
\label{trivclass}
The 2-groupoid $\triv{\mathbb{G}}$ has the structure of a module over the monoidal 2-groupoid $\ugrb M$ of bundle gerbes over $M$. Moreover:
\begin{enumerate}
\item[(i)]
If $\mathbb{T}=(\mathcal{S},\mathcal{A},\sigma)$ is a trivialization of $\mathbb{G}$, and $\mathcal{K}$ is a bundle gerbe over $M$, the new trivialization $\mathbb{T}.\mathcal{K}$ has the bundle gerbe $\mathcal{S} \otimes \pi^{*}\mathcal{K}$.   

\item[(ii)]
On isomorphism classes,  a free and transitive action of the group
$\hc 0 \ugrb M$ on the set $\hc 0 \triv{\mathbb{G}}$ is induced.
\end{enumerate}
\end{lemma}

The definition of this action and the proof of its properties is an exercise in bundle gerbe theory and deferred to Section \ref{bundlegerbesact}. 

Now we can complete the proof of Theorem \ref{class}.

\begin{proposition}
\label{stringbijection}
Let $P$ be a principal $\spin n$-bundle over $M$. Then, the assignment
\begin{eqnarray*}
\bigset{4.2cm}{Isomorphism classes of trivializations of $\mathbb{CS}_P$} &\to&\bigset{3.5cm}{String classes on $P$}
\\[\medskipamount]
(\mathcal{S},\mathcal{A},\sigma)&\mapsto&\mathrm{DD}(\mathcal{S})
\end{eqnarray*}
is a bijection.
\end{proposition}

\begin{proof}
After what we have said about the 1-morphisms between trivializations, the assignment $(\mathcal{S},\mathcal{A},\sigma) \mapsto \mathrm{DD}(\mathcal{S})$ does not depend on the choice of a representative. With Proposition \ref{stringstructureaustriv}, it is hence well-defined. It is a bijection because it is an equivariant map between $\h^3(M,\Z)$-torsors:
domain and codomain are torsors due to Lemma \ref{trivclass} (ii) and Theorem \ref{stringclassi} (b). The equivariance follows from Lemma \ref{trivclass} (i). 
\end{proof}

\begin{remark}
One can regard a trivialization of a bundle 2-gerbe $\mathbb{G}$ as a \quot{twisted  gerbe}, analogous to the notion of a \emph{twisted line bundle}, which is nothing but a trivialization of a gerbe \cite[Section 3]{waldorf1}. In \cite[Section 6]{Fiorenza} this analogy is extended to \emph{twisted string structures}.
\end{remark}

\section{Connections on Bundle 2-Gerbes}  

\label{con2gerbes}

We provide definitions of connections on bundle 2-gerbes and on their trivializations, introduce our notion of a string connection and outline the proofs of our results on string connections.

\subsection{Connections on Chern-Simons 2-Gerbes}

\label{sec:conncs}

We prove Theorem \ref{cancon}: for $\mathbb{CS}_P$,  the Chern-Simons 2-gerbe associated to a principal $\spin n$-bundle $P$, we construct  the connection $\nablaa$ on $\mathbb{CS}_P$ associated to  a connection $A$ on $P$.

Again, I only want to recall a few facts about connections on bundle gerbes at this place. Basic definitions are provided in Section \ref{cons}, and more detail is provided in \cite{stevenson1,waldorf1}. An important feature of a connection $\nabla$ on a bundle gerbe $\mathcal{G}$ is its curvature, which is a closed 3-form $\mathrm{curv}(\nabla) \in \Omega^3(M)$ and represents the Dixmier-Douady class $\mathrm{DD}(\mathcal{G})$ in real cohomology. Bundle gerbes with connection form another 2-groupoid $\ugrbcon M$, which has a forgetful 2-functor
\begin{equation*}
\ugrbcon M \to \ugrb M\text{.}
\end{equation*} 
As already mentioned in Section \ref{sumstringcon} one has to be aware that for a 1-morphism to be \quot{connection-preserving} is \emph{structure}, not property; see Definition \ref{conpres}. Thus, the 1-morphisms of $\ugrbcon M$ are \emph{isomorphisms with compatible connections}. The 2-morphisms are \emph{connection-preserving transformations}; here it is only property. 

\begin{definition}
\label{twogerbecon}
 Let $\mathbb{G} = (Y,\mathcal{P},\mathcal{M},\mu)$ be a bundle 2-gerbe over $M$. A \emph{connection} on $\mathbb{G}$ is a 3-form $B \in \Omega^3(Y)$, together with a connection $\nabla$ on $\mathcal{P}$ of curvature 
\begin{equation}
\label{condcurv}
\mathrm{curv}(\nabla)=\pi_2^{*}B-\pi_1^{*}B\text{,}
\end{equation}
and a compatible connection on the product $\mathcal{M}$, such that the associator $\mu$ is connection-preserving. 
\end{definition}

Analogously to bundle gerbes with connection, every connection on a bundle 2-gerbe has a curvature. It is the unique 4-form $F \in \Omega^4(M)$ such that $\pi^{*}F = \mathrm{d}B$. It is closed and its cohomology class coincides with the image of $\mathrm{CC}(\mathbb{G})$ in real cohomology \cite[Proposition 8.2]{stevenson2}.

We recall from Section \ref{cstwogerbe} that the Chern-Simons 2-gerbe $\mathbb{CS}_P(\mathcal{G})$ is specified by two parameters: the principal $G$-bundle $P$ and a  multiplicative bundle gerbe $\mathcal{G}$ over $G$. It turns out that a  connection on $\mathbb{CS}_P(\mathcal{G})$ is determined by three parameters: a connection $A$ on $P$, an invariant bilinear form $\left \langle -,-  \right \rangle$ on the Lie algebra $\mathfrak{g}$ of $G$, and a certain kind of connection on the multiplicative bundle gerbe $\mathcal{G}$ \cite{waldorf5}. 

In the case  $G=\spin n$, we choose the bilinear form to be the Killing form on $\mathfrak{spin}(n)$, normalized such that the closed 3-form
\begin{equation}
\label{canthreeform}
H=\textstyle{\frac{1}{6}}\left \langle \theta \wedge [\theta \wedge \theta]  \right \rangle \in \Omega^3(\spin n)\text{,}
\end{equation}
with $\theta$ the left-invariant Maurer-Cartan form,
represents the image of the standard generator $1\in\h^3(\spin n,\Z)$ in real cohomology.

Let us briefly recall the construction of connections on Chern-Simons 2-gerbes $\mathbb{CS}_P(\mathcal{G})$ following \cite{johnson1,carey4}.   
 The first ingredient, the 3-form on $P$, is the \emph{Chern-Simons 3-form}
\begin{equation*}
TP(A) := \left \langle  A \wedge \mathrm{d}A \right \rangle + {\textstyle\frac{2}{3}}\left \langle A \wedge A \wedge A  \right \rangle \in \Omega^3(P)\text{.}
\end{equation*}
The reason to choose this particular 3-form becomes clear when one computes the difference between its two pullbacks to $P^{[2]}$, 
\begin{equation*}
\pi_2^{*}TP(A) - \pi_1^{*}TP(A) = g^{*}H + \mathrm{d}\omega\text{,}
\end{equation*}
with $H$ the 3-form \erf{canthreeform} and a certain 2-form $\omega  \in \Omega^2(P^{[2]})$. 
According to condition \erf{condcurv}, this difference has to coincide with the curvature of the connection on the bundle gerbe $\mathcal{P} =g^{*}\mathcal{G}$ of $\mathbb{CS}_P(\mathcal{G})$.

The basic gerbe $\mathcal{G}$ of Gaw\c edzki-Reis and Meinrenken comes with a  connection of curvature $H$. So, the pullback gerbe $\mathcal{P}=g^{*}\mathcal{G}$ carries a connection of curvature $g^{*}H$. This connection can further be modified using the 2-form $\omega$ (see Lemma \ref{affine}), such that the resulting connection $\nabla_{\omega}$ on $\mathcal{P}$ has the desired  curvature $g^{*}H + \mathrm{d}\omega$. Thus, the connection $\nabla_{\omega}$ on $\mathcal{P}$ qualifies as the second ingredient of the connection on $\mathbb{CS}_P$.

We continue the construction of the connection $\nablaa$ following \cite{waldorf5}, where the remaining steps have been reduced to the problem of finding a (multiplicative) connection on the multiplicative structure on $\mathcal{G}$ satisfying a certain curvature constraint. Such connections are unique (up to an action of closed, simplicially closed 1-forms on $G \times G$, which does not change the isomorphism class of the multiplicative gerbe $\mathcal{G}$ with connection); see \cite[Corollary 3.1.9]{waldorf5}. Even better, the construction of the multiplicative structure on the basic gerbe described in Section \ref{cstwogerbe} produces exactly such a connection.

Thereby, we have the connection $\nablaa$ on $\mathbb{CS}_P$.

\begin{lemma}[{\cite[Theorem 3.12 (b)]{waldorf5}}]
\label{cscurv}
The curvature of the connection $\nablaa$ on the Chern-Simons 2-gerbe $\mathbb{CS}_P$ is the 4-form
\begin{equation*}
\left \langle  \Omega_A \wedge \Omega_A \right \rangle \in \Omega^4(M)\text{,}
\end{equation*}
where $\Omega_A$ is the curvature of $A$.
\end{lemma}

Since the curvature of any bundle 2-gerbe represents its characteristic class, which is  here $\frac{1}{2}p_1(P)$, it follows that the curvature of $\nablaa$ is one-half of the Pontryagin 4-form of $P$. (The factor  $\frac{1}{2}$  does not appear in the formula, because we have normalized the bilinear form $\left \langle -,-  \right \rangle$ with respect to $\spin n$ and not with respect to $\so n$.)

\subsection{Connections and Geometric String Structures}

\label{contriv}

We give the details of our new notion of a string connection (Definition \ref{stringcon}) as a compatible connection on a trivialization of the Chern-Simons 2-gerbe $\mathbb{CS}_P$. Then we prove Theorem \ref{threeform}, the relation between geometric string structures and 3-forms.

\begin{definition}
\label{defcompconn}
Let $\mathbb{G}$ be a bundle 2-gerbe over $M$ with connection, and let $\mathbb{T}=(\mathcal{S},\mathcal{A},\sigma)$ be a trivialization of $\mathbb{G}$. A \emph{compatible connection} $\babla$ on $\mathbb{T}$ is a connection on the bundle gerbe $\mathcal{S}$ and a  compatible connection on the isomorphism $\mathcal{A}$, such that the transformation $\sigma$ is connection-preserving.
\end{definition}

The 3-form $K_{\babla}$ we have mentioned in Section \ref{sumstringcon} is the curvature of the connection on the bundle gerbe $\mathcal{S}$. Under the bijection of Proposition \ref{stringbijection}, it thus represents the string class $\xi_{\mathbb{T}} = \mathrm{DD}(\mathcal{S})$ in real cohomology.

First we generalize Lemma \ref{trivvanish} to the setup with connection.

\begin{lemma}
\label{geomstrex}
Let $\mathbb{G}$ be a bundle 2-gerbe with connection. Then, $\mathrm{CC}(\mathbb{G})=0$ if and only if $\mathbb{G}$ admits a trivialization with compatible connection.
\end{lemma}

It is clear that if a trivialization with compatible connection is given, then $\mathrm{CC}(\mathbb{G})=0$ by Lemma \ref{trivvanish}.
We prove the converse in Section \ref{trivcon} using the fact that bundle 2-gerbes are classified by degree four differential cohomology.

Concerning the algebraic structure of trivializations with connection, we have the following lemma.

\begin{lemma}
\label{catgeomstr}
Let $\mathbb{G}$ be a bundle 2-gerbe with connection $\nabla$ and $\mathrm{CC}(\mathbb{G})=0$. 
\begin{enumerate}
\item[(i)] 
The trivializations of $\mathbb{G}$ with compatible connection form a 2-groupoid $\triv {\mathbb{G},\nabla}$. 

\item[(ii)]
The 2-groupoid $\triv{\mathbb{G},\nabla}$ is a module for the monoidal 2-groupoid $\ugrbcon M$ of bundle gerbes with connection over $M$.

\item[(iii)]
On isomorphism classes, a free and transitive action of the group $\hc 0 \ugrbcon M$ on the set $\hc 0 \triv {\mathbb{G},\nabla}$ is induced.

\end{enumerate}
\end{lemma}

This lemma is a straightforward generalization of our previous results (Lemmata \ref{trivgroupoid} and \ref{trivclass}) on trivializations without connection in the sense that all constructions are literally the same; see Remarks \ref{catgens} and \ref{actgens}.  

We  recall from  \cite[Theorem 4.1]{murray2} that bundle gerbes with connection are classified by degree three differential cohomology,
\begin{equation}
\label{grbconclass}
\hc 0 \ugrbcon M \cong \hat\h^3(M,\Z)\text{.}
\end{equation}
Under this identification, the two maps
\begin{equation*}
\mathrm{pr}: \hat \h^3(M,\Z) \to \h^3(M,\Z)
\quad\text{ and }\quad
\Omega: \hat \h^3(M,\Z) \to \Omega_{\mathrm{cl}}^3(M)
\end{equation*}
are given by the Dixmier-Douady class of the bundle gerbe, and the curvature of its connection, respectively. Together with Lemma \ref{catgeomstr}, this implies Corollaries \ref{geomstrcat} and \ref{geomstrtorsor}.

Now we are heading towards the proof of Theorem \ref{threeform}, the relation between geometric string structures and 3-forms.

\begin{lemma}
\label{3form}
Let $\mathbb{G}$ be a bundle 2-gerbe with covering $\pi:Y \to M$ and a connection with 3-form $B\in\Omega^3(Y)$. Let $\mathbb{T}=(\mathcal{S},\mathcal{A},\sigma)$ be a trivialization of $\mathbb{G}$ with compatible connection $\babla$, the connection on the bundle gerbe $\mathcal{S}$ denoted by $\nabla$. Then, there exists a unique 3-form $H_{\babla}$ on $M$ such that
\begin{equation*}
\pi^{*}H_{\babla} = \mathrm{curv}(\nabla) + B\text{.}
\end{equation*}
\end{lemma}

\begin{proof}
We can prove this right away.  For the 3-form 
\begin{equation*}
C:= \mathrm{curv}(\nabla) + B \in \Omega^3(Y)\text{,}
\end{equation*}
we compute
\begin{eqnarray*}
\pi_2^{*}C- \pi_1^{*}C &=& (\pi_2^{*}\mathrm{curv}(\nabla) - \pi_1^{*}\mathrm{curv}(\nabla)) + (\pi_2^{*}B - \pi_1^{*}B)
\\&=& \mathrm{curv}(\mathcal{P}) - \mathrm{curv}(\mathcal{P})
\\&=& 0\text{.}
\end{eqnarray*}
Here we have used condition \erf{condcurv} and that the isomorphism
\begin{equation*}
\mathcal{A}: \mathcal{P} \otimes \pi_2^{*}\mathcal{S} \to \pi_1^{*}\mathcal{S}
\end{equation*}
is equipped with a compatible connection (which implies the equality of the curvatures of the target and the source bundle gerbes).
The  computation means (see the exact sequence in Lemma \ref{exact}), that $C$ is the pullback of a unique 3-form $H_{\babla}$ along $\pi:Y \to M$. 
\end{proof}

The 3-form $H_{\babla}$ depends only on the isomorphism class of $\mathbb{T}$ in $\triv {\mathbb{G},\nabla}$, since the curvature of $\mathcal{S}$ does so. Thus, the proof of Theorem \ref{threeform} is finished with the following calculation, which follows directly from the definitions.

\begin{lemma}
\label{3form2}
Let $\mathbb{G}$ be a bundle 2-gerbe with connection, and let $\mathbb{T}$ be a trivialization with connection $\babla$. The 3-form $H_{\babla}$ has the following properties:
\begin{enumerate}
\item
$\mathrm{d}H_{\babla} = \mathrm{curv}(\mathbb{G})$,
\item
$H_{\babla.\mathcal{K}} = \mathrm{curv}(\mathcal{K}) +  H_{\babla}$,
\end{enumerate}
where $\mathcal{K}$ is a bundle gerbe with connection over $M$, and $\babla.\mathcal{K}$ the connection on $\mathbb{T}.\mathcal{K}$ for the action from Lemma \ref{catgeomstr} (ii).
\end{lemma}

\subsection{Existence and Classification of compatible Connections on Trivializations}

\label{exclass}

Concerning our results on string connections, we are left with the proof of  Theorem \ref{stringconcon}. In fact a more general statement is true for  trivializations of any bundle 2-gerbe. The first part is as follows.

\begin{proposition}
\label{trivconex}
Suppose $\mathbb{T}$ is a trivialization of a bundle 2-gerbe $\mathbb{G}$, and suppose $\nabla$ is a connection on $\mathbb{G}$. Then, there exists a connection $\blacktriangledown$ on $\mathbb{T}$ compatible with $\nabla$. \end{proposition}

For $\mathbb{G}=\mathbb{CS}_P$, Proposition \ref{trivconex} implies the first part of Theorem \ref{stringconcon}. Its proof requires two lemmata.

\begin{lemma}[{\cite[Section 6]{murray}}]
\label{conex}
Let $\mathcal{K}$ be a bundle gerbe over $M$. Then, $\mathcal{K}$ admits a connection.
\end{lemma}

\begin{lemma}
\label{conind}
Let $\mathbb{G}$ be a bundle 2-gerbe, let $\mathbb{T}_1$ and $\mathbb{T}_2$ be trivializations and let $\mathbb{B}:\mathbb{T}_1 \to \mathbb{T}_2$ be a 1-morphism. Let $\nabla$ be a connection on $\mathbb{G}$, and let $\babla$ be a connection on $\mathbb{T}_2$ compatible with $\nabla$. Then, there exists a connection on $\mathbb{T}_1$ compatible with $\nabla$ such that $\mathbb{B}$ becomes a 1-morphism in $\triv{\mathbb{G},\nabla}$. 
\end{lemma}

The proof of this lemma is carried out in Section \ref{conpullback}.
Now we give the proof of the above proposition.\medskip 

\begin{proofblank}{Proof of Proposition \ref{trivconex}} 
Since $\mathbb{G}$ has by assumption the trivialization $\mathbb{T}$, we have $\mathrm{CC}(\mathbb{G})=0$ by Lemma \ref{trivvanish}. Hence, by Lemma \ref{geomstrex},  there exists a trivialization $\mathbb{T}'$ of $\mathbb{G}$ with connection $\babla'$ compatible with $\nabla$. Of course $\mathbb{T}'$ is not necessarily equal or isomorphic to the given trivialization $\mathbb{T}$, but by Lemma \ref{trivclass} (ii) the  isomorphism classes of trivializations of $\mathbb{G}$ form a torsor over $\hc 0 \ugrb M$. Thus, there exists a bundle gerbe $\mathcal{K}$ over $M$ and a 1-morphism $\mathbb{B}:\mathbb{T} \to \mathcal{K}.\mathbb{T}'$. According to Lemma \ref{conex} every bundle gerbe admits a connection; so we may choose one on $\mathcal{K}$. Now we use the action of Lemma \ref{catgeomstr} (ii) according to which $\mathcal{K}.\mathbb{T}'$ also has a compatible connection. Finally, by Lemma \ref{conind}, the 1-morphism $\mathbb{B}$ induces a compatible connection on $\mathbb{T}$.
\end{proofblank}

Now we want to describe the space  of compatible connections on a fixed trivialization $\mathbb{T}=(\mathcal{S},\mathcal{A},\sigma)$ of a bundle 2-gerbe $\mathbb{G}$ with a covering $\pi:Y \to M$. In order to make the following statements, we have to infer that part of the structure of the bundle gerbe $\mathcal{S}$ over $Y$ is another covering $\omega:W \to Y$ (see Definition \ref{bundlegerbe}). The following vector space $V_{\mathbb{T}}$ associated to the trivialization $\mathbb{T}$ will be relevant. It is the quotient
\begin{equation*}
V_{\mathbb{T}} := \left ( \Omega^2(M) \oplus \Omega^1(Y) \oplus \Omega^1(W)  \right ) \;/\; U \text{,}
\end{equation*}
where the linear subspace $U$ we divide out is given by
\begin{equation*}
U := \left \lbrace (\mathrm{d}\chi, \pi^{*}\chi + \nu,\omega^{*}\nu) \;|\; \chi\in\Omega^1(M)\text{, }\nu\in\Omega^1(Y) \right \rbrace\text{.}
\end{equation*}
Now, the second part of Theorem \ref{stringconcon} is implied by the following proposition.

\begin{proposition}
\label{trivconaff}
Suppose $\mathbb{T}=(\mathcal{S},\mathcal{A},\sigma)$ is a trivialization of  $\mathbb{G}$, and suppose $\nabla$ is a connection on $\mathbb{G}$. Then, the set of compatible connections on $\mathbb{T}$ is an affine space over $V_{\mathbb{T}}$.
\end{proposition}

The proof of Proposition \ref{trivconaff} is  technical, and most of the work is deferred to Sections \ref{action1} and \ref{action2}. At this place I want to at least give a hint as to why the vector space $V_{\mathbb{T}}$ appears. We require the following two lemmata. The first describes how to act on the set of connections of a fixed bundle gerbe.
\begin{lemma}
\label{affine}
The set of connections on a bundle gerbe $\mathcal{S}$ over $Y$ is an affine space over the real vector space
\begin{equation*}
V_{\mathcal{S}} := \left ( \Omega^2(Y) \oplus \Omega^1(W)  \right ) / \left ( \mathrm{d} \oplus \omega^{*}  \right )\Omega^1(Y)\text{,}
\end{equation*}
where $\omega:W \to Y$ is the covering of $\mathcal{S}$. 
\end{lemma}

We prove this lemma in Section \ref{cons} based on results of Murray. The second lemma describes how to act on the set of connections on an isomorphism. Here we have to infer that  an isomorphism also comes with its own covering. 

\begin{lemma}
\label{affineiso}
Let $\mathcal{A}:\mathcal{G} \to \mathcal{H}$ be an isomorphism between bundle gerbes with covering $Z$. Then, the set of connections on $\mathcal{A}$ is an affine space over $\Omega^1(Z)$. 
\end{lemma}

Notice that this is a statement on the set of connections that are \emph{not} necessarily compatible with connections on the bundle gerbes $\mathcal{G}$ and $\mathcal{H}$. Its  proof can also be found in Section \ref{cons}.\medskip

\begin{proofblank}{Proof of Proposition \ref{trivconaff}}
Let us describe the action of the vector space $V_{\mathbb{T}}$ on the set of (not necessarily compatible) connections on $\mathbb{T}=(\mathcal{S},\mathcal{A},\sigma)$.
For $(\psi,\rho,\varphi)\in V_{\mathbb{T}}$, consider the pair $(\eta,\varphi)\in V_{\mathcal{S}}$ with $\eta := \mathrm{d}\rho - \pi^{*}\psi \in \Omega^2(Y)$. It operates on the connection on the bundle gerbe $\mathcal{S}$ according to Lemma \ref{affine}. For $Z$ the covering space of the 1-isomorphism $\mathcal{A}: \mathcal{P} \otimes \pi_2^{*}\mathcal{S} \to \pi_1^{*}\mathcal{S}$ of bundle gerbes over $Y^{[2]}$, $Z$ has a projection $p: Z \to W \times_M W$. Consider 
\begin{equation*}
\varepsilon :=  p^{*}(\delta(\omega^{*}\rho-\varphi)) \in \Omega^1(Z)\text{,}
\end{equation*}
where $\delta$ is the linear map
\begin{equation*}
\delta := \omega_2^{*} - \omega_1^{*}: \Omega^1(W) \to \Omega^1(W \times_M W)\text{,}
\end{equation*}
(cf. Lemma \ref{exact}). The 1-form $\varepsilon$ operates on the 1-isomorphism $\mathcal{A}$ according to Lemma \ref{affineiso}.

It is straightforward to check that this action is well-defined under dividing out the subvectorspace $U$:
suppose we have 1-forms $\chi\in\Omega^1(M)$ and $\nu\in\Omega^1(Y)$ and act by the triple consisting of $\psi := \mathrm{d}\chi$, $\rho := \pi^{*}\chi + \nu$ and $\varphi := \omega^{*}\nu$. It follows that $\eta = \mathrm{d}\nu$ so that $(\eta,\varphi)\in V_{\mathcal{S}}$ acts trivially by Lemma \ref{affine}. Furthermore, we find $\omega^{*}\rho-\varphi= \omega^{*}\pi^{*}\chi$, so that its alternating sum $\delta$ vanishes (see again Lemma \ref{exact}). Hence, $\varepsilon=0$. 

The remaining steps are the content of the following lemma, which is to be proven in Sections \ref{action1} and \ref{action2}.
\end{proofblank}

\begin{lemma}
\label{lemaction1}
\label{lemaction2}
The action of $V_{\mathbb{T}}$ on connections on $\mathbb{T}$ has the following properties:
\begin{enumerate}
\item[(a)]
It takes compatible connections to compatible connections.

\item[(b)]
It is free and transitive on compatible connections.
\end{enumerate}
\end{lemma}

\setsecnumdepth{1}

\section{Trivializations of Chern-Simons Theory}

\label{trivcs}

In this section we  compare the geometric string structures introduced in Definition \ref{stringcon} with the   concept introduced by Stolz and Teichner:

\begin{definition}[{\cite[Definition 5.3.4]{stolz1}}]
\label{defstringst}
A \emph{geometric string structure} on a principal $\spin n$-bundle $P$  with connection $A$ over $M$ is a trivialization of the extended Chern-Simons theory $Z_{P,A}$. [...]
\end{definition}

In Section \ref{csex} we explain what the extended Chern-Simons theory $Z_{P,A}$ is in a generic model of \quot{$n$-bundles with connection}. In Section \ref{models} we mention three possible models, and discuss in detail the one formed by \quot{bundle $n$-gerbes with connection}. In Section \ref{trivcsdet} we discuss trivializations of $Z_{P,A}$ and prove Theorem \ref{stringtrivcs}, which states a certain  equivalence between Definition \ref{defstringst} and our Definition \ref{stringcon}. Before going on, I have to mention two aspects of Definition \ref{defstringst} that are not completely covered in \cite{stolz1}. 
\begin{enumerate}
\item[(a)]
The first aspect concerns smoothness conditions that have to be imposed on Chern-Simons theories and their trivializations. Such conditions are mentioned is the last sentence of \cite[Definition 5.3.4]{stolz1} (that is what \quot{[...]} in Definition \ref{defstringst} refers to):
\begin{equation}
\newlength{\parboxwidth}
\addtolength{\parboxwidth}{\textwidth}
\addtolength{\parboxwidth}{-3.4cm}
\parbox{\parboxwidth}{
\quot{\textit{... these data fit together to give bundles (respectively sections in these bundles) over the relevant mapping spaces.}}}
\label{quote1}
\end{equation}
The structure group of some of these bundles  is the group $\mathrm{Out}(A)$ of outer automorphisms of a certain von Neumann algebra --- this group is not a Lie group (not even infinite-dimensional), so that it is not totally obvious in which sense these bundles can be smooth. Below, we will interpret the above quote as close as possible in the context of bundle $n$-gerbes with connection.

I remark that Stolz, Teichner and Hohnhold recently came up with a rigorous treatment of smoothness for  field theories (and supersymmetry) in terms of categories fibred over manifolds (resp. supermanifolds) \cite{hohnhold1,stolz5}.

\item[(b)]
The second aspect concerns the values of Chern-Simons theories and their trivializations on manifolds with boundaries or corners. Here \cite[Definition 5.3.4]{stolz1} says: \quot{\textit{There are also data associated to manifolds with boundary and these data must fit together when gluing manifolds and connections.}} Recent results of Hopkins and Lurie \cite{lurie1} suggest that a rigorous formulation of such gluing conditions must be based on $(\infty,n)$-categories of cobordisms. However, these results are newer than the paper \cite{stolz1} (and in fact, to some extent, emerged from it). For that reason we are going to ignore manifolds with corners or boundaries in the following.  
\end{enumerate}

In contrast to these  issues with the pioneering definition of Stolz and Teichner, Definition \ref{stringcon} that we propose in this article  \emph{is} complete. In particular, a theorem that literally states the equivalence between the two definitions cannot be expected, and I think that Theorem \ref{stringtrivcs} that we prove below is the best approximation one can have.

\setsecnumdepth{2}

\subsection{Chern-Simons Theory as an Extended 3d TFT}

\label{csex}

\def\nbuncon#1#2{#1\text{-}\ubuncon{#2}}

At first sight, the biggest difference between the formalism of Stolz and Teichner and ours is a different geometric model \quot{$n$-bundles with connection} --- geometrical objects classified by the differential cohomology group $\hat \h^{n+1}(M,\Z)$ mentioned in Section \ref{resstringcon}. We consider this difference as unessential for the following constructions, and start by talking generically about \emph{$n$-bundles with connection} as geometrical representatives for classes in $\hat \h^{n+1}(M,\Z)$. We assume  the following minimal requirements for this abstract model: 
\begin{enumerate}

\item[(i)]

$n$-bundles with connection over $M$ form a category $\nbuncon nM$. In particular, this category could be the homotopy 1-category of some $n$-category.
Moreover, the categories  $\nbuncon nM$ form  a presheaf over smooth manifolds, i.e. $n$-bundles with connection can consistently be pulled back along smooth maps. In particular, we have a category  $\mathcal{C}_n := \nbuncon n*$ of \quot{fibres}, and for every point $x\in M$ a functor
\begin{equation*}
\iota_x^{*}: \nbuncon nM \to \mathcal{C}_n
\end{equation*}
is induced by the inclusion $\iota_x$ of $x$ into $M$.

\item[(ii)]
For $X^d$ a closed oriented $d$-dimensional smooth manifold, 
there is a \quot{transgression} functor 
\begin{equation*}
\mathscr{T}_{X^d}: \nbuncon n{M} \to \nbuncon {(n-d)}{C^{\infty}(X^d,M)}\text{.}
\end{equation*}
This functor covers, on the level of differential cohomology, the usual transgression homomorphism, which can for instance be treated in terms of Deligne cohomology \cite{gomi2}.
For $d=0$ and $X^0$ a single point, we require $\mathscr{T}_{X^0}$ to coincide with the pullback along the \quot{evaluation map} $\ev: C^{\infty}(X^0,M) \to M$. 
\end{enumerate}

In this abstract setting, an \emph{$n$-dimensional extended topological field theory $Z$ over $M$}
assigns to every  closed oriented smooth manifold $X^d$ of dimension $0\leq d \leq n$ an $(n-d)$-bundle $Z(X^d)$ with connection over $C^{\infty}(X^d,M)$. This is neither (a) the most general nor (b) a complete definition. Concerning (a), it actually only includes \emph{classical} field theories.
Concerning (b), one would additionally require  relations  over  manifolds with boundaries --- as explained above, we decided to ignore these.

In spite of that,   the above notion of an $n$-dimensional extended TFT is not completely naive: it implements exactly  the smoothness condition \erf{quote1} of Stolz and Teichner. This becomes explicit by saying that the \emph{value} of a TFT $Z$ at a point $\phi\in C^{\infty}(X^d,M)$ is the object  object $\iota_{\phi}^{*}Z(X^d)$ in the category $\mathcal{C}_{d-n}$ of fibres. By construction, these values form a bundle over the mapping space $C^{\infty}(X^d,M)$.

\emph{Chern-Simons theories over $M$} are examples of  extended 3-dimensional TFTs. They are parameterized by    $3$-bundles with connection over $M$, where the  Chern-Simons theory $Z_F$ associated to a 3-bundle $F$ with connection is defined by transgression,
\begin{equation}
\label{eq:defcs}
Z_F(X^d) := \mathscr{T}_{X^d}(F)\text{.}
\end{equation}
In the following, we  specify a concrete model for $n$-bundles with connection, and then use definition \erf{eq:defcs} to compare our concept of geometric string structures with the one of Stolz and Teichner.

\begin{remark}
In the classical picture of Chern-Simons theory one has  $M=BG$. In this case, the 3-bundle $F$ with connection represents a \emph{level} $k \in \hat \h^4(BG,\Z)$.
The maps $\phi:X^d \to BG$   classify principal $G$-bundles over $X^d$. Under these identifications, \erf{eq:defcs} is Freed's original concept  of Chern-Simons theory in its extended version \cite[Section 3.2]{freed3}. 
In order to avoid certain problems that come from the fact that $BG$ is not a smooth manifold, one inserts the \quot{target manifold} $M$, factors the maps $\phi: X^d \to BG$ through a fixed map $\xi: M \to BG$ and replaces the class $\xi^{*}k \in \hat\h^4(M,\Z)$ by a fixed 3-bundle $F$ with connection --- this is the above point of view. The same strategy is also applied in \cite{stolz1}.
\end{remark}

\subsection{Models for $n$-Bundles with Connection} 

\label{models}

We mention three different models for \quot{$n$-bundles with connection} that could furnish the categories $\nbuncon nX$ and the transgression functors $\mathscr{T}_{X^d}$. 
\begin{enumerate}

\item[(a)]
The first model is the \emph{Hopkins-Singer model} \cite[Definition 2.5]{hopkins1} that Freed uses for his definition of Chern-Simons theory \cite{freed3}. It can be seen as a categorical version of Cheeger-Simons differential characters, where  cocycles form the objects and coboundaries the morphisms.

\item[(b)]
Stolz and Teichner understand an $n$-bundle over $X$ as a (at least continuous) map $f:X \to B_n$, where $B_n$ is a topological space playing the role of a $K(\Z,n+1)$ \cite[page 78]{stolz1}: 
\begin{center}
\begin{tabular}{l||l|l|l|p{4.2cm}}
$n$   & 0 & 1 & 2 & 3 \\\hline
$B_n$ & $S^1$ & $PU(A)$ & $\mathrm{Out}(A)$ & space of $\mathrm{Out}(A)$-torsors
\end{tabular}
\end{center} 
Here, $A$ is a type $\text{III}_1$ factor. 
The morphisms are homotopies between these maps. As remarked above, it is not totally clear what the smoothness assumptions on these maps are. Clearly, a $0$-bundle must be a \emph{smooth} map $f:X \to S^1$. In the next instance, it is  possible to specify what \emph{smooth} maps into $PU(A)$ are, in such a way that an equivalence with the category of principal $S^1$-bundles over $X$ is obtained. Another point is that there are no connections in this picture and, accordingly,  transgression (and thus, the Chern-Simons theory) cannot completely be defined in this setting \cite[Remark 5.3.2]{stolz1}.

\item[(c)]
The third model is \quot{bundle $n$-gerbes with connection}, where --- unfortunately --- an $n$-bundle corresponds to an $(n-1)$-gerbe; see Figure  \ref{fig:real}.
\begin{figure}
\begin{center}
\begin{tabular}{|p{1cm}|p{5.7cm}|p{3.8cm}|}\hline
$n$ & $n$-bundle with connection &  over a point \\\hline\hline
0   & smooth function $f:X \to \ueins$ & element of $\ueins$ \\\hline
1   & principal $\ueins$-bundle with connection & $\ueins$-torsor \\\hline
2   & bundle gerbe with connection & $\ueins$-groupoid \\\hline
3   & bundle 2-gerbe with connection & (omitted) \\\hline
\end{tabular}
\end{center}
\caption{Our model of $n$-bundles with connection. }
\label{fig:real}
\end{figure}
 All definitions are available (see Definitions \ref{twogerbe}, \ref{twogerbecon}, \ref{bundlegerbe}, \ref{bundlegerbecon}), together with discussions of their categorical structure, e.g. in Section \ref{trivcon}. The  transgression functors $\mathscr{T}_{X^d}$ can be defined completely, although not all definitions have appeared in the literature and here we will limit ourselves to the relevant aspects.  

\end{enumerate}

From that point on, we use the model (c) for $n$-bundles with connection,  because (A)  in that model we have the Chern-Simons 2-gerbe constructed in Sections \ref{cstwogerbe} and \ref{sec:conncs}, and (B) because  of the existing discussions of transgression functors in this context. Let us first explain the relevant aspects of it. The categories $\mathcal{C}_n$ of fibres can be derived by looking at bundle $n$-gerbes over a point --- this is left as an exercise.

The transgression functors $\mathscr{T}_{X^d}$ are a little bit more involved, but only the cases $d=0,3$ will be important later and these are the easiest ones. Let $\mathbb{G}$ be a bundle 2-gerbe with connection.  
\begin{itemize}
\item
For $d=3$, the smooth function $\mathscr{T}_{X^3}(\mathbb{G}): C^{\infty}(X^3,M) \to \ueins$ is the holonomy of $\mathbb{G}$. Its value at a map $\phi:X^3 \to M$ is obtained by choosing a trivialization $\mathbb{T}$ of $\phi^{*}\mathbb{G}$ with compatible connection $\babla$, and integrating the associated 3-form $H_{\babla}$ from Lemma \ref{3form} over $X^3$:
\begin{equation}
\label{holdef}
\mathscr{T}_{X^3}(\mathbb{G})(\phi) := \exp \left ( \int_{X^3} H_{\babla} \right )\text{.}
\end{equation}
If $\mathbb{H}$ is an isomorphic bundle 2-gerbe with connection, its holonomy coincides with the one of $\mathbb{G}$, and so $\mathscr{T}_{X^3}(\mathbb{G})= \mathscr{T}_{X^3}(\mathbb{H})$     as required since there are only identity morphisms in the category of smooth $\ueins$-valued functions.

\item
In the case $d=2$ we are concerned with a closed oriented surface $X^2\equiv \Sigma$. We will follow the  strategy of   transgressing a \emph{bundle gerbe} to the loop space  \cite[Section 3.1]{waldorf5}. The fibre of the principal $\ueins$-bundle $\mathscr{T}_{\Sigma}(\mathbb{G})$ over a map $\phi$ consists of 1-isomorphism classes of trivializations of $\phi^{*}\mathbb{G}$ with compatible connections, i.e.
\begin{equation*}
\mathscr{T}_{\Sigma}(\mathbb{G})|_{\phi} := \hc 0 \trivcon {\phi^{*}\mathbb{G}}\text{.} \end{equation*}
These fibres have the structure of  $\ueins$-torsors in virtue of Lemma \ref{catgeomstr} (iii) and the identifications $\hc 0 \ugrbcon \Sigma \cong \h^2(\Sigma,\ueins) \cong \ueins$. Analogously to \cite[Proposition 3.1.2]{waldorf5} one can show that this yields a Fréchet principal $\ueins$-bundle over $C^{\infty}(\Sigma,M)$. 

The connection on $\mathscr{T}_{\Sigma}(\mathbb{G})$ is defined by prescribing its parallel transport using tools developed jointly with Schreiber \cite{schreiber3,waldorf9} analogously to \cite[Section 4.2]{waldorf10}. All this is functorial: if $\mathbb{A}: \mathbb{G} \to \mathbb{H}$ is a morphism with compatible connection, one gets a bundle morphism
\begin{equation*}
\mathscr{T}_{\Sigma}(\mathbb{A}) : \mathscr{T}_{\Sigma}(\mathbb{G}) \to \mathscr{T}_{\Sigma}(\mathbb{H}) : [\mathbb{T}] \mapsto [\mathbb{T} \circ \phi^{*}\mathbb{A}^{-1}]
\end{equation*}
that regards a trivialization  as a morphism $\mathbb{T}:\phi^{*}\mathbb{G} \to \mathbb{I}$ to the trivial bundle 2-gerbe $\mathbb{I}$ (see \erf{trivasmorph}) and pre-composes it with the inverse of $\mathbb{A}$.

\item
The case $d=1$ and $X^1 \equiv S^1$ can be treated in a simple way by making the assumption that the surjective submersion $\pi:Y \to M$ of the bundle 2-gerbe $\mathbb{G}$ is \quot{loopable}, i.e. we assume that the map $L\pi: LY \to LM$ is again a surjective submersion. The assumption is satisfied in the case of the Chern-Simons 2-gerbe $\mathbb{CS}_P$, since its submersion is the projection of a principal bundle with connected structure group (see e.g. \cite[Proposition 1.9]{spera1} and \cite[Lemma 5.1]{waldorf13}). 
Our main input is the well-established transgression of bundle gerbes loop spaces, discussed in detail in \cite[Section 3.1]{waldorf5}.
This transgression is a monoidal functor $\mathscr{T}$ and commutes with pullbacks along smooth maps.

The bundle gerbe $\mathscr{T}_{S^1}(\mathbb{G})$ over $LM$ has the surjective submersion $L\pi$, over $LY^{[2]}$ it has the principal $\ueins$-bundle $P := \mathscr{T}(\mathcal{P})$ with connection and over $LY^{[3]}$ it has the isomorphism
\begin{equation*}
\mathscr{T}(\mathcal{M}): L\pi_{12}^{*}P \otimes L\pi_{23}^{*}P \to L\pi_{13}^{*}P\text{.}
\end{equation*}
 The associator $\mu$ of $\mathbb{G}$ transgresses to the associativity condition for the bundle gerbe product $\mathscr{T}(\mathcal{M})$. The curving 3-form of $\mathscr{T}_{S^1}(\mathbb{G})$ is simply the transgression of the curving 2-form of $\mathbb{G}$. In the same straightforward manner, any 1-isomorphism $\mathbb{G} \to \mathbb{H}$ with compatible connection transgresses to a 1-isomorphism $\mathscr{T}_{S^1}(\mathbb{G}) \to \mathscr{T}_{S^1}(\mathbb{H})$ with compatible connection.

\item
For $d=0$, we put $\mathscr{T}_{X^0} = \ev^{*}$ as required.

\end{itemize} 

With the above definitions of the transgression functors $\mathscr{T}_{X^d}$, and definition \erf{eq:defcs} of a Chern-Simons theory we make the following definition.  

\begin{definition}
\label{def:cs}
Let $P$ be a principal $\spin n$-bundle over $M$ with connection $A$, and let  $F=(\mathbb{CS}_P,\nablaa)$ be the Chern-Simons 2-gerbe together with its connection $\nablaa$. Then, 
\begin{equation*}
Z_{P,A} := Z_{F}\text{,}
\end{equation*}
is called the \emph{Chern-Simons theory associated to $(P,A)$}. 
\end{definition}

In the next subsection, we will apply Stolz-Teichner's definition of a trivialization to \emph{this} version of Chern-Simons theory.  
Just in order to verify that Definition \ref{def:cs} is correct, we get the following consequence of Theorem \ref{threeform}.

\begin{corollary}
The value of the Chern-Simons theory $Z_{P,A}$ on a closed oriented 3-manifold $\phi: X^3 \to M$ is the classical Chern-Simons invariant:
\begin{equation*}
\iota_{\phi}^{*}Z_{P,A}(X^3) = \exp \left ( \int_{X^3} s^{*}TP(A)  \right )\text{,}
\end{equation*}
where $TP(A) \in \Omega^3(P)$ is the Chern-Simons 3-form and $s:X^3 \to P$ is a section along $\phi$.
\end{corollary} 

\begin{proof}
The section exists because $\spin n$ is simply connected. To prove the formula,
one combines \erf{holdef} and \erf{threeformdiff} with the fact that the 3-form $K_{\babla}$ in \erf{threeformdiff} has integral class and thus vanishes under the integral.
\end{proof}

The same result has been proved in \cite{carey4}. In that paper,  the relation between the bundle 2-gerbe $\mathbb{CS}_P$ and Chern-Simons theory has originally been established  \cite[Theorem 6.7]{carey4}.

\subsection{Sections of $n$-Bundles and Trivializations}

\label{trivcsdet}

According to Stolz and Teichner, a trivialization of an extended TFT consists of \quot{sections} into the respective  $n$-bundles with connections; see   \erf{quote1}. Next we will describe what a  section of an $n$-bundle with connection is in terms of our model, bundle $n$-gerbes with connection. The results are summarized in Figure \ref{fig:triv}.

\begin{figure}
\begin{center}
\begin{tabular}{|p{0.2cm}|p{3.1cm}|p{5cm}|p{3.4cm}|}\hline
$n$ & Object & Section & over a point $x$ \\\hline\hline
0 &  function $f$ & $\R$-valued smooth function that exponentiates to $f$ & $t\in \R$ with\hfill\mbox{} $\exp(2\pi\im t)= f(x)$ \\\hline
1& $\ueins$-bundle $P$ & smooth section of $P$  & point $t\in P_x$ \\\hline
2& bundle gerbe $\mathcal{G}$ & trivialization of $\mathcal{G}$ with compatible connection & Morita equivalence  $\mathcal{G}_x \cong B\ueins$ \\\hline
3& bundle 2-gerbe $\mathbb{G}$ &  trivialization of $\mathbb{G}$ with  compatible connection & (omitted) \\\hline
\end{tabular}
\end{center}
\caption{Sections of $n$-bundles with connection. }
\label{fig:triv}
\end{figure}

The case $n=0$ is somewhat exceptional --- also in \cite{stolz1}.  There, a section of a smooth function $f:X \to \ueins$ is by definition a smooth function  $g:X \to \R$ such that $f = \exp(2\pi\im g)$. For $n=1$, the prescription is literally true: a section of a principal $\ueins$-bundle $P$ over $X$ with connection is just a (smooth) section. It is important to notice that this section is \emph{not} required to be flat.

For $n=2,3$ it is not a priori clear what the correct notion of a section of a bundle gerbe or a bundle 2-gerbe is.  We have to impose  additional constraints. A plausible condition is that sections are certain morphisms in the categories $\nbuncon nX$ and that the transgression functors $\mathscr{T}_{X^d}$  send sections to sections. Indeed, for the cases $n=0,1$ above this is true: consider a section $\sigma$ of a principal $\ueins$-bundle $P$ with connection $\omega$. It can be considered as a connection-preserving morphism between $P$ and the trivial bundle $\trivlin_{\sigma^{*}\omega}$ equipped with the connection 1-form $\sigma^{*}\omega$. Transgressing such a morphism gives
\begin{equation*}
\mathrm{Hol}_{\omega}(\tau) = \exp \left ( 2\pi\im \int_{\tau} \sigma^{*}\omega \right )
\end{equation*}
for all loops $\tau\in LM$, which means exactly that the function $g(\tau) := \int_{\tau} \sigma^{*}\omega$  is a section of $\mathscr{T}_{S^1}(P,\omega) = \mathrm{Hol}_{\omega}$.

The requirement that sections are morphisms in $\nbuncon nX$ enforces us to define a \emph{section of a bundle $n$-gerbe with connection} to be a  trivialization with compatible connection. Indeed, for a bundle gerbe $\mathcal{G}$ with connection, such a trivialization is the same as an isomorphism $\mathcal{T}:\mathcal{G} \to \mathcal{I}_{\rho}$ with compatible connection (see Definition \ref{conpres}), for a 2-form $\rho$, as noticed in \cite[Section 3.1]{waldorf1}. For a bundle 2-gerbe, an analogous statement is explained in Section \ref{trivcon}; see \erf{trivasmorph}. Since we have introduced transgression as a \emph{functor}, and trivial bundle $n$-gerbes transgress to trivial bundle $(n-1)$-gerbes, it follows that sections transgress to sections. We have summarized our notion of sections in Figure \ref{fig:triv}.

\begin{definition}
\label{def:trivcs}
Let $\mathbb{G}$ be a bundle 2-gerbe with connection over $M$ and let $Z_{\mathbb{G}}$ be the extended Chern-Simons theory over $M$ in the sense of \erf{eq:defcs}. A \emph{trivialization of $Z_{\mathbb{G}}$} assigns to each oriented closed manifold $X^d$ of dimension $0\leq d\leq n$ a section $\mathcal{S}(X^d)$ of $Z_{\mathbb{G}}(X^d)$. 
\end{definition}

Naturally, in a more advanced version of extended TFTs trivializations would show a specific behaviour on manifolds with boundaries.  Also, we have to take isomorphism classes of trivializations because sections of bundle $n$-gerbes are morphisms in an $n$-groupoid, whereas the values $Z_{\mathbb{G}}(X^d)$ are objects in the (truncated) category $\nbuncon {(n-d)}{X^d}$. 
We obtain an obvious map
\begin{equation*}
\mathscr{S}: \bigset{4.3cm}{Isomorphism classes of trivializations of $\mathbb{G}$ with compatible connection} \to
\bigset{5.4cm}{Isomorphism classes of trivializations of the extended Chern-Simons theory $Z_{\mathbb{G}}$}
\end{equation*}
defined by
\begin{equation*}
\mathscr{S}(\mathbb{T})(X^d) := \mathscr{T}_{X^d}(\mathbb{T})\text{.}
\end{equation*}
This is the map  announced in Section \ref{sumstringcon} that compares (for $\mathbb{G}=(\mathbb{CS}_P$,$\nablaa$)) our notion of geometric string structures with the one of Stolz and Teichner in Theorem \ref{stringtrivcs}.  The first part of the proof of Theorem \ref{stringtrivcs} is as follows.

\begin{lemma}
\label{lem:inj}
The map $\mathscr{S}$ is injective. 
\end{lemma}

\begin{proof}
Let $\mathbb{T}_1$ and $\mathbb{T}_2$ be two trivializations with compatible connections, and assume that $\mathscr{S}(\mathbb{T}_1)=\mathscr{S}(\mathbb{T}_2)$. Over the point $X^0$, we have for $k=0,1$\text{,}
\begin{equation*} 
\mathscr{S}(\mathbb{T}_k)(X^0) = \mathscr{T}_{X^0}(\mathbb{T}_k) = \ev^{*}\mathbb{T}_k\text{.}
\end{equation*}
So the equality $\mathscr{S}(\mathbb{T}_1)(X^0) = \mathscr{S}(\mathbb{T}_2)(X^0)$ implies that the isomorphism classes of $\ev^{*}\mathbb{T}_1$ and $\ev^{*}\mathbb{T}_2$ are equal. Since $\ev$ is a diffeomorphism, this implies $\mathbb{T}_1 \cong \mathbb{T}_2$. 
\end{proof}

Finally, we want to formulate an assumption under which the map $\mathscr{S}$ is also surjective. We say that the \emph{cobordism hypothesis holds for the Chern-Simons theory $Z_\mathbb{G}$} if $Z_{\mathbb{G}}$ and all  trivializations $\mathcal{S}$ of $Z_{\mathbb{G}}$ are determined by their value on the point. That means, in particular, if $\mathcal{S}_1$ and $\mathcal{S}_2$ are trivializations of $Z_{\mathbb{G}}$ such that $\mathcal{S}_1(X^0)=\mathcal{S}_2(X^0)$, then $\mathcal{S}_1=\mathcal{S}_2$.

Whether or not the  cobordism hypothesis does hold for $Z_{\mathbb{G}}$  cannot be answered unless the exact assignments and conditions are formulated, which $Z_{\mathbb{G}}$ has for manifolds with boundaries and corners. As mentioned above, this has been done neither here nor in  \cite{stolz1}. However, I want to bring up two arguments that might make our assumption plausible.
\begin{enumerate}
\item 
Whichever way the behaviour of $Z_{\mathbb{G}}$ for manifolds with boundary is formulated, it should fit into Lurie's definition of extended topological field theories \cite[Definition 1.2.13]{lurie1}. But in that context, Lurie  \emph{proved} the cobordism hypothesis \cite[Theorem 1.2.16]{lurie1}.

\item
A 3-dimensional extended TFT $Z$ in the sense of Lurie with $Z(X^0) = \mathbb{G}$ exists. This also follows from the cobordism hypothesis \cite[Theorem 1.2.16]{lurie1} and the fact that every bundle 2-gerbe $\mathbb{G}$ with connection is a \emph{fully dualizable object} in the $3$-groupoid of bundle 2-gerbes with connection (it is even \emph{invertible}, indicating that the corresponding TFT is a classical one).

\end{enumerate}
The main purpose of the assumption that the cobordism hypothesis holds for $Z_{\mathbb{G}}$ is to provide a basis for the following.

\begin{lemma}
\label{lem:surj}
Under the assumption that the cobordism hypothesis holds for the extended Chern-Simons theory $Z_{\mathbb{G}}$, the map $\mathscr{S}$ is surjective. \end{lemma}

\begin{proof}
Suppose $\mathcal{S}$ is a trivialization of $Z_{\mathbb{G}}$. Consider its value $\mathcal{S}(X^0)$, which is a section of $\mathscr{T}_{X^0}(\mathbb{G})=\ev^{*}\mathbb{G}$. Since $\ev$ is a diffeomorphism, we may put $\mathbb{T} := (\ev^{-1})^{*}\mathcal{S}(X^0)$ which is a section of $\mathbb{G}$, i.e. a trivialization with compatible connection. Then, $\mathscr{S}(\mathbb{T}) = \mathcal{S}$. 
\end{proof}

\section{Background on Bundle Gerbes}

\label{background}

This section introduces some of the basic definitions concerning bundle gerbes and connections on bundle gerbes on the basis of \cite{waldorf1}. There are a few new results about  spaces of connections on bundle gerbes and isomorphisms.

\subsection{Bundle Gerbes}

\label{bundlegerbes}

Let $M$ be a smooth manifold. We recall that a \emph{covering} is a surjective submersion, and we refer to Section \ref{cstwogerbe} for our conventions concerning fibre products and the labelling of projections.  

\begin{definition}[{\cite[Section 3]{murray}}]
\label{bundlegerbe}
A \emph{bundle gerbe} over $M$ is a covering $\pi:Y \to M$ together with a principal $\ueins$-bundle $P$ over $Y^{[2]}$ and a bundle isomorphism
\begin{equation*}
\mu: \pi_{12}^{*}P \otimes \pi_{23}^{*}P \to \pi_{13}^{*}P
\end{equation*}
over $Y^{[3]}$ which is associative over $Y^{[4]}$.
\end{definition}

The notion of an isomorphism between bundle gerbes took some time to develop; most appropriate for our purposes is the following generalization of a \quot{stable isomorphism}.
We consider two bundle gerbes $\mathcal{G}_1$ and $\mathcal{G}_2$ over $M$, whose structure is denoted in the same way as in Definition \ref{bundlegerbe} but with indices $1$ or $2$.

\begin{definition}[{\cite[Definition 2]{waldorf1}}]
\label{defiso}
An \emph{isomorphism} $\mathcal{A}:\mathcal{G}_1 \to \mathcal{G}_2$ is a covering $\zeta: Z \to Y_1 \times_M Y_2$ together with a principal $\ueins$-bundle $Q$ over $Z$ and a bundle isomorphism
\begin{equation}
\label{isoalpha}
\alpha: P_1 \otimes \zeta_2^{*}Q \to \zeta_1^{*}Q \otimes P_2
\end{equation}
over $Z \times_M Z$, which satisfies a compatibility condition with $\mu_1$ and $\mu_2$. 
\end{definition}

In order to make the notation less complicated we fix the convention that we suppress writing pullbacks along  maps whenever it is clear which map is meant. For example, in  \erf{isoalpha} the bundles $P_1$ and $P_2$ are understood to be pulled back along the evident maps $Z \times_M Z \to Y_k^{[2]}$. 

The first example of an isomorphism is the identity isomorphism $\id_{\mathcal{G}}$ of a bundle gerbe $\mathcal{G}$. It has the identity covering $\zeta:=\id_{Y^{[2]}}$, the principal bundle $Q := P$ and an isomorphism $\alpha$ defined from the isomorphism $\mu$ in a straightforward way. 

Suppose that $\mathcal{A}_1$ and $\mathcal{A}_2$ are two isomorphisms from $\mathcal{G}_1$ to $\mathcal{G}_2$. 
\begin{definition}[{\cite[Definition 3]{waldorf1}}]
\label{def:trans}
A \emph{transformation}  $\beta: \mathcal{A}_1 \Rightarrow \mathcal{A}_2$ is a covering \begin{equation*}
k: V \to Z_1 \times_{(Y_1 \times_M Y_2)} Z_2
\end{equation*}
together with a bundle isomorphism $\beta_V: Q_1 \to Q_2$ between the pullbacks of the bundles of $\mathcal{A}_1$ and $\mathcal{A}_2$ to $V$, which satisfies a compatibility condition with the isomorphisms $\alpha_1$ and $\alpha_2$.
\end{definition}

Additionally, transformations $(V_1,\beta_{V_1})$ and $(V_2,\beta_{V_2})$ are identified whenever the bundle isomorphisms $\beta_{V_1}$ and $\beta_{V_2}$ agree after being pulled back to the fibre product of $V_1$ and $V_2$. 

All the operations that turn bundle gerbes, isomorphisms and transformations into a monoidal 2-groupoid are straightforward to find. For example, the tensor product of two bundle gerbes $\mathcal{G}_1$ and $\mathcal{G}_2$ over $M$ has the covering $Z := Y_1 \times_M Y_2 \to M$ and over $Z^{[2]}$ the principal $\ueins$-bundle $P_1 \otimes P_2$, where again our convention of suppressing evident maps from the notation is employed. The tensor unit is the \emph{trivial bundle gerbe} $\mathcal{I}$, which has the identity covering $\id_M$, the trivial principal $\ueins$-bundle, and whose isomorphism $\mu$  is fibrewise multiplication. For a full treatment, we refer to Section 1 of \cite{waldorf1}.

In Section \ref{bundlegerbesact} we also need \emph{duals} of bundle gerbes. To every bundle gerbe $\mathcal{G}$ one assigns a \emph{dual bundle gerbe} $\mathcal{G}^{*}$, and similar to every isomorphism $\mathcal{A}:\mathcal{G} \to \mathcal{H}$ a \emph{dual isomorphism} $\mathcal{A}^{*}: \mathcal{H}^{*} \to \mathcal{G}^{*}$, and to every transformation $\beta: \mathcal{A}_1 \Rightarrow \mathcal{A}_2$ a \emph{dual transformation} $\beta^{*}: \mathcal{A}_2^{*} \Rightarrow \mathcal{A}_1^{*}$. Complete definitions can be found in \cite[Section 1.3]{waldorf1}. Basically, the dual bundle gerbe $\mathcal{G}^{*}$ has the dual principal $\ueins$-bundle $P^{*}$ (i.e. the same set but with $\ueins$ acting by inverses). The dual isomorphism $\mathcal{A}^{*}$ has the \emph{same} principal $\ueins$-bundle $Q$ as before, and the dual transformation also has the same bundle isomorphism as before.

There is an   isomorphism $\mathcal{D}_{\mathcal{G}}: \mathcal{G}^{*} \otimes \mathcal{G} \to \mathcal{I}$ that expresses that $\mathcal{G}^{*}$ is dual to $\mathcal{G}$.  It has the identity covering $\zeta := \id_{Y^{[2]}}$, the dual bundle $Q:=P^{*}$, and its isomorphism is  defined from the bundle isomorphism $\mu$; see \cite[Section 1.2]{waldorf4}. In Section \ref{bundlegerbesact} we need two properties of $\mathcal{D}_{\mathcal{G}}$.  The first is the existence of a  transformation
\begin{equation*}
\alxydim{@R=1.7cm}{\mathcal{G}^{*} \otimes \mathcal{G} \ar[dr]_{\mathcal{D}_{\mathcal{G}}}="1" \ar[rr]^-{\mathcal{A}^{*-1} \otimes \mathcal{A}} && \mathcal{H}^{*} \otimes \mathcal{H} \ar@{=>}"1"|-*+{\varphi_{\mathcal{A}}} \ar[dl] ^{\mathcal{D}_{\mathcal{H}}} \\ & \mathcal{I}&}
\end{equation*}
associated to every isomorphism $\mathcal{A}:\mathcal{G} \to \mathcal{H}$. It can be seen as a \quot{naturality} property of $\mathcal{D}_{\mathcal{G}}$.  The second property is the existence of a  transformation between the isomorphism
\begin{equation*}
\alxydim{@C=1.5cm}{\mathcal{G} \ar[r]^-{\id \otimes \mathcal{D}_{\mathcal{G}}} & \mathcal{G} \otimes \mathcal{G}^{*} \otimes \mathcal{G} \ar[r]^-{\mathcal{D}_{\mathcal{G}^{*}} \otimes \id} & \mathcal{G}}
\end{equation*}
and the identity isomorphism $\id_{\mathcal{G}}$. It can be seen as a \quot{zigzag} property for $\mathcal{D}_{\mathcal{G}}$. The definitions of these two transformations are easy to find once one has the one of $\mathcal{D}_{\mathcal{G}}$ at hand. 

\subsection{Connections}

\label{cons}

Let $\mathcal{G}=(Y,P,\mu)$ be a bundle gerbe over $M$.

\begin{definition}[{\cite[Section 6]{murray}}]
\label{bundlegerbecon}
A \emph{connection} on $\mathcal{G}$ is a 2-form $C \in \Omega^2(Y)$ and a connection $\omega$ on $P$ of curvature $\mathrm{curv}(\omega)= \pi_2^{*}C - \pi_1^{*}C$, such that $\mu$ is connection-preserving.
\end{definition}

We have already mentioned  that every bundle gerbe admits a connection (see Lemma \ref{conex}). One can see this by choosing some connection on $P$, some 2-form $C$, and then  correcting them using the following result of Murray.
\begin{lemma}[{\cite[Section 8]{murray}}]
\label{exact}
Let $\pi:Y \to M$ be a surjective submersion and $p\geq 0$ be an integer. Consider \begin{equation*}
\delta_\pi :=  \sum_{i=1}^{k} (-1)^{k} \pi_{1,...,i-1,i+1,...,k}^{*}: \Omega^p(Y^{[k-1]}) \to \Omega^p(Y^{[k]})\text{.}
\end{equation*}
Then, the sequence
\begin{equation*}
\alxydim{@C=1.2cm}{0 \ar[r] & \Omega^p(M) \ar[r]^{\pi^{*}=\delta_{\pi}} & \Omega^p(Y) \ar[r]^-{\delta_\pi} & \Omega^p(Y^{[2]}) \ar[r]^-{\delta_\pi} & \hdots}
\end{equation*}
is exact.
\end{lemma}

For example, we see that for a connection $(C,\omega)$ on a bundle gerbe $\mathcal{G}$,
\begin{equation*}
\delta_{\pi}\mathrm{d}C = \pi_2^{*}\mathrm{d}C - \pi_1^{*}\mathrm{d}C = \mathrm{d}\mathrm{curv}(\omega)=0\text{.}
\end{equation*}
This means that $\mathrm{d}C$ is the pullback of a unique 3-form $H \in \Omega^3(M)$. This 3-form is the \emph{curvature} of the connection $(C,\omega)$.
We can now  give the following proof. 

\begin{proofblank}{Proof of Lemma \ref{affine} from Section \ref{exclass}}
We have to show that the set of connections on $\mathcal{G}$ is an affine space over the vector space
\begin{equation*}
V_{\mathcal{G}} := \left ( \Omega^2(M) \oplus \Omega^1(Y)  \right ) / \left ( \mathrm{d} \oplus \pi^{*}  \right )\Omega^1(M)\text{.}
\end{equation*}
We shall first define the action and then show that it is free and transitive.
\begin{enumerate}
\item 
If $(C,\omega)$ is a connection on $\mathcal{G}$, and $(\eta,\varphi)\in V_{\mathcal{G}}$, we have a new connection $(C',\omega')$ defined by $C' := C + \mathrm{d}\varphi - \pi^{*}\eta$ and $\omega' := \omega + \delta_\pi\varphi$. Here we recall our convention according to which $\delta\varphi$ is understood to be pulled back to the total space of $P$. Indeed, 
\begin{equation*}
\mathrm{curv}(\omega') = \mathrm{curv}(\omega) + \mathrm{d}\delta_\pi\varphi =\delta _\pi(C + \mathrm{d}\varphi) = \delta_\pi(C' + \pi^{*}\eta)=\delta_\pi C'\text{,}
\end{equation*}
which is the first condition. Further, since $\delta_\pi(\delta_\pi\varphi)=0$, the isomorphism $\mu$ over $Y^{[3]}$ preserves $\omega'$.
This shows that $(C',\omega')$ is again a connection on $\mathcal{G}$. It is clear that this action is well-defined on the quotient $V_{\mathcal{G}}$. 

\item

The action  is free: suppose $(\eta,\varphi) \in V_{\mathcal{G}}$ acts trivially. Then, $\delta_\pi\varphi=0$ so that there exists $\nu\in\Omega^1(M)$ with $\pi^{*}\nu=\varphi$ by Lemma \ref{exact}. Further, $\mathrm{d}\varphi - \pi^{*}\eta = 0$ which implies $\pi^{*}(\mathrm{d}\nu - \eta)=0$. Since $\pi^{*}$ is injective by Lemma \ref{exact}, $\eta=\mathrm{d}\nu$. Thus, $(\eta,\varphi)=0$ in the quotient space $V_{\mathcal{G}}$.

\item

The action is transitive: suppose $(C,\omega)$ and $(C',\omega')$ are connections. Then, there is a 1-form $\psi\in\Omega^1(Y^{[2]})$ such that $\omega' = \omega + \psi$. Since $\mu$ is connection-preserving for both $\omega$ and $\omega'$, we see that $\delta_\pi\psi=0$. By Lemma \ref{exact}, there is a 1-form $\varphi\in\Omega^1(Y)$ such that $\delta_\pi\varphi=\psi$. We compute
\begin{equation*}
\delta_\pi(C'-C - \mathrm{d}\psi) = \mathrm{curv}(\omega') - \mathrm{curv}(\omega) - \delta_\pi\mathrm{d}\psi = \mathrm{d}\psi -  \mathrm{d}\delta_\pi\psi = 0\text{,}
\end{equation*}
which means that there exists $\eta\in\Omega^2(M)$ with $-\pi^{*}\eta=C'-C-\mathrm{d}\psi$. Then, the action of $(\eta,\varphi)\in V_{\mathcal{G}}$  takes $(C,\omega)$ to $(C',\omega')$. 
\end{enumerate}
We also remark that the action of an element $(\eta,\varphi)\in V_{\mathcal{G}}$ changes the curvature of $\mathcal{G}$ by $\mathrm{d}\eta$.
\end{proofblank}

Next is the discussion of connections on isomorphisms. 

\begin{definition}
\label{conpres}
Let $\mathcal{A}:\mathcal{G}_1 \to \mathcal{G}_2$ be an isomorphism between bundle gerbes over $M$. 
\begin{enumerate}
\item 
A \emph{connection} on $\mathcal{A}$ is a connection $\kappa$ on its principal $\ueins$-bundle $Q$. 

\item
Suppose $\nabla_1=(C_1,\omega_1)$ and $\nabla_2=(C_2,\omega_2)$ are connections on $\mathcal{G}_1$ and $\mathcal{G}_2$. Then, a connection $\kappa$ on $\mathcal{A}$ is called \emph{compatible} with $\nabla_1$ and $\nabla_2$, if 
$\mathrm{curv}(\kappa) = \zeta^{*}(C_2 - C_1)$ and $\alpha$ is connection-preserving.

\end{enumerate}
\end{definition}

Now we see immediately the claim of Lemma \ref{affineiso}: the (non-compatible) connections on $\mathcal{A}$ are an affine space over the 1-forms on $Z$.  

The next lemma shows how to pullback a bundle gerbe connection along an isomorphism. 

\begin{lemma}
\label{conpull}
Suppose $\mathcal{A}:\mathcal{G}_1 \to \mathcal{G}_2$ is a 1-isomorphism, and $\nabla_2$ is a connection on $\mathcal{G}_2$. Then, there exists a connection $\kappa$ on $\mathcal{A}$ and a connection $\nabla_1$ on $\mathcal{G}_1$, such that $\kappa$ is compatible with $\nabla_1$ and $\nabla_2$.
\end{lemma}

\begin{proof}
We choose any connection $(C_1,\omega_1)$ on the bundle gerbe $\mathcal{G}_1$, and any connection $\kappa$ on $\mathcal{A}$. We compare the pullback connection $\alpha^{*}(\zeta_1^{*}\kappa + \omega_2)$ with the connection $\omega_1 + \zeta_2^{*}\kappa$ on $P_1 \otimes \zeta_2^{*}Q$. They differ by a 1-form $\beta\in\Omega^1(Z^{[2]})$. The condition on $\alpha$ implies that $\delta \beta=0$ over $Z^{[3]}$, so that there exists a 1-form $\gamma\in\Omega^1(Z)$ with $\beta=\delta\gamma$. Operating with $\gamma$ on $\kappa$, we obtain a new connection $\kappa'$ on $\mathcal{A}$ such that $\alpha$ is connection-preserving. Consider now the 2-form
\begin{equation*}
B := \mathrm{curv}(\kappa') - C_2 + C_1 \in \Omega^2(Z)\text{.}
\end{equation*}  
One readily computes $\delta B=0$ over $Z^{[2]}$ so that there exists a 2-form $\eta \in \Omega^2(M)$ such that $B = \pi^{*}\eta$. Operating with $\eta$ on the connection $(C_1,\omega_1)$ yields a new connection on $\mathcal{G}_1$ for which $\kappa'$ is compatible.
\end{proof}

Next, we need to pull back connections on isomorphisms along transformations.
We say that a transformation $\beta: \mathcal{A} \Rightarrow \mathcal{A}'$ between isomorphisms with connections is connection-preserving, if the bundle isomorphism $\beta_V$ is connection-preserving. 

\begin{lemma}
\label{compconpullback}
Let $\mathcal{G}_1$ and $\mathcal{G}_2$ be bundle gerbes with connection, $\mathcal{A}:\mathcal{G}_1 \to \mathcal{G}_2$ and $\mathcal{A}':\mathcal{G}_1 \to \mathcal{G}_2$ be  isomorphisms, and $\beta: \mathcal{A} \Rightarrow \mathcal{A}'$ is a transformation. Suppose $\kappa'$ is a compatible connection on $\mathcal{A}'$.  Then, there exists a unique compatible connection on $\mathcal{A}$ such that $\beta$ is connection-preserving.
\end{lemma}

\begin{proof}
Suppose $\kappa$ is a compatible connection on $\mathcal{A}$. We consider the pullback connection $k^{*}\kappa$ on $k^{*}Q$ for the projection $k: V \to Z \times_{(Y_1 \times_M Y_2)} Z'$ from Definition \ref{def:trans}. The assumption that $\beta$ is connection-preserving is $k^{*}\kappa = \beta^{*}k^{*}\kappa'$, and since $k$ is a surjective submersion, this fixes $\kappa$ uniquely (see Lemma \ref{exact}). To prove the existence, we define a connection $\tilde\kappa := \beta^{*}k^{*}\kappa'$ on $k^{*}Q$. Then, the compatibility of $\kappa'$ together with the relation between the isomorphism $\beta$ and the isomorphisms $\alpha$ and $\alpha'$ that belong to $\mathcal{A}$ and $\mathcal{A}'$, respectively, show that $\tilde\kappa$ descends to a compatible connection $\kappa$ on $\mathcal{A}$.
\end{proof}

Finally, let me mention that all additional structure on the 2-groupoid $\ugrb M$ that we have listed in the previous section, also exists for the 2-groupoid $\ugrbcon M$  formed by bundle gerbes over $M$ with connection, isomorphisms with compatible connections and connection-preserving transformations; see \cite{waldorf1}.

\section{Technical Details}

\label{tech}

In this section we provide  the proofs of all remaining lemmata.

\subsection{Lemma \ref{trivgroupoid}: Trivializations form a 2-Groupoid}

\label{triv2cat}

We prove Lemma \ref{trivgroupoid}: we define a 2-groupoid $\triv{\mathbb{G}}$ of trivializations of a bundle 2-gerbe $\mathbb{G}$. The bundle 2-gerbe $\mathbb{G}$ may consist of a covering $\pi:Y \to M$, a bundle gerbe $\mathcal{P}$ over $Y^{[2]}$, a product $\mathcal{M}$ over $Y^{[3]}$ and an associator $\mu$ over $Y^{[4]}$. 

We recall from Definition \ref{deftriv} that a trivialization $\mathbb{T}$ of $\mathbb{G}$ consists of a bundle gerbe $\mathcal{S}$ over $Y$, an isomorphism $\mathcal{A}: \mathcal{P} \otimes \pi_2^{*}\mathcal{S} \to \pi_1^{*}\mathcal{S}$ over $Y^{[2]}$, and a transformation $\sigma$. Given   trivializations $\mathbb{T}=(\mathcal{S},\mathcal{A},\sigma)$ and $\mathbb{T}'=(\mathcal{S}',\mathcal{A}',\sigma')$ of $\mathbb{G}$, a \emph{1-morphism} $\mathbb{B}: \mathbb{T} \to \mathbb{T}'$ is an isomorphism $\mathcal{B}: \mathcal{S} \to \mathcal{S}'$ between bundle gerbes over $Y$ together with a transformation
\begin{equation*}
\alxydim{@=1.1cm}{\mathcal{P} \otimes \pi_2^{*}\mathcal{S} \ar[r]^-{\mathcal{A}} \ar[d]_{\id \otimes \pi_2^{*}\mathcal{B}} & \pi_1^{*}\mathcal{S} \ar@{=>}[dl]|*+{\beta} \ar[d]^{\pi_1^{*}\mathcal{B}} \\ \mathcal{P} \otimes \pi_2^{*}\mathcal{S}' \ar[r]_-{\mathcal{A}'} & \pi_1^{*}\mathcal{S}'}
\end{equation*} 
over $Y^{[2]}$ which is compatible with the transformations $\sigma$ and $\sigma'$ in the sense of the pentagon diagram shown in Figure \ref{compmorph}.

\begin{figure}[t]
\begin{footnotesize}
\begin{equation*}
\def\ausgl{0.15cm}
\alxydim{@C=-3.7cm@R=2.5cm}{&&\pi_{12}^{*}\mathcal{A}' \circ (\id \otimes \pi_2^{*}\mathcal{B}) \circ (\id \otimes \pi_{23}^{*}\mathcal{A}) \ar@{=>}[dll]_-*+{\pi_{12}^{*}\beta \circ \id} \ar@{=>}[drr]^-*+{\id \circ \pi_{23}^{*}\beta}&&\\\hspace{\ausgl}\pi_1^{*}\mathcal{B} \circ \pi_{12}^{*}\mathcal{A}' \circ (\id \otimes \pi_{23}^{*}\mathcal{A})\hspace{\ausgl} \ar@{=>}[dr]_-*+{\id \circ \sigma} &&&&\pi_{12}^{*}\mathcal{A}' \circ (\id \otimes \pi_{23}^{*}\mathcal{A}') \circ (\id \otimes \id \otimes \pi_3^{*}\mathcal{B}) \ar@{=>}[dl]^-*+{\sigma' \circ \id}\\& \pi_1^{*}\mathcal{B} \circ (\id \otimes \pi_{13}^{*}\mathcal{A}) \circ (\mathcal{M} \otimes \id)  \ar@{=>}[rr]_-*+{\pi_{13}^{*}\beta \circ \id} &\hspace{9cm}& \pi_{13}^{*}\mathcal{A}' \circ (\mathcal{M} \otimes \pi_3^{*}\mathcal{B})\hspace{\ausgl}&}
\end{equation*}
\end{footnotesize}
\caption{The compatibility between the transformations $\sigma$ and $\sigma'$ of two trivializations $\mathbb{T}$ and $\mathbb{T}'$ and the transformation $\beta$ of a 1-morphism $\mathbb{B}=(\mathcal{B},\beta)$ between $\mathbb{T}$ and $\mathbb{T}'$. It is an equality of transformations over $Y^{[3]}$.}
\label{compmorph}
\end{figure}

The identity 1-morphism as well as the composition between 1-morphisms are straightforward to find using the structure of the 2-groupoid of bundle gerbes.

If $\mathbb{B}_1=(\mathcal{B}_1,\beta_1)$ and $\mathbb{B}_2=(\mathcal{B}_2,\beta_2)$ are both 1-morphisms between $\mathbb{T}$ and $\mathbb{T}'$, a \emph{2-morphism} is a transformation $\varphi: \mathcal{B}_1 \Rightarrow \mathcal{B}_2$ which is compatible with the transformations $\beta_1$ and $\beta_2$ in such a way  that the diagram
\begin{equation*}
\alxydim{@=1.3cm}{\pi_1^{*}\mathcal{B}_1 \circ \mathcal{A} \ar@{=>}[d]_{\pi_1^{*}\varphi \circ \id} \ar@{=>}[r]^-{\beta_1} & \mathcal{A}' \circ (\id \otimes \pi_2^{*}\mathcal{B}_1) \ar@{=>}[d]^{\id \circ (\id \otimes \pi_2^{*}\varphi)} \\ \pi_1^{*}\mathcal{B}_2 \circ \mathcal{A} \ar@{=>}[r]_-{\beta_2} & \mathcal{A}' \circ (\id \otimes \pi_2^{*}\mathcal{B}_2)}
\end{equation*}
is commutative. Horizontal and vertical composition of 2-morphisms are the ones of the 2-groupoid of bundle gerbes.

It is clear that every 2-morphism is invertible, since every transformation is invertible. In the same way, every 1-morphism is invertible (up to 2-morphisms), since every 1-isomorphism between bundle gerbes is invertible up to a transformation.

The axioms of the 2-groupoid $\triv {\mathbb{G}}$ can easily be deduced from those of the 2-groupoid of bundle gerbes. 

\begin{remark}
\label{catgens}
If $\nabla$ is a connection on the bundle 2-gerbe $\mathbb{G}$, it is straightforward to repeat the above definitions in the 2-category of bundle gerbe \emph{with connection}. Explicitly, all bundle gerbes are equipped with connections, all isomorphisms with compatible connections, and all transformations are connection-preserving. The result is again a 2-groupoid $\triv{\mathbb{G},\nabla}$ whose objects are the trivializations with connection compatible with $\nabla$.
\end{remark}

\subsection{Lemma \ref{trivclass}: Bundle Gerbes act on Trivializations}

\label{bundlegerbesact}

We  exhibit the 2-groupoid $\triv {\mathbb{G}}$ of trivializations of a bundle 2-gerbe $\mathbb{G}$ over $M$ as a module for the 2-category $\ugrb M$ of bundle gerbes over $M$. The module structure is a strict 2-functor
\begin{equation}
\label{action}
\triv {\mathbb{G}} \times \ugrb M \to \triv {\mathbb{G}}
\end{equation}
satisfying the usual axioms in a strict way.

We remark that all results of this section generalize analogous results for an action of principal $\ueins$-bundles over $M$ on trivializations of bundle (1-)gerbes;  see \cite[Theorem 2.5.4]{waldorf4} and references therein.

If $\mathbb{T}=(\mathcal{S},\mathcal{A},\sigma)$ is a trivialization and $\mathcal{K}$ is a bundle gerbe over $M$, we obtain a new trivialization $\mathbb{T}.\mathcal{K}$ consisting of the bundle gerbe $\mathcal{S} \otimes \pi^{*}\mathcal{K}$ over $Y$. Since $\pi_1^{*}\pi^{*}\mathcal{K} = \pi_2^{*}\pi^{*}\mathcal{K}$, its isomorphism is simply
\begin{equation*}
\mathcal{A} \otimes \id: \mathcal{P} \otimes \pi_2^{*}\mathcal{S} \otimes \pi_2^{*}\pi^{*}\mathcal{K} \to \pi_1^{*}\mathcal{S} \otimes \pi_1^{*}\pi^{*}\mathcal{K}\text{.}
\end{equation*}
In the same way, its transformation is $\sigma \otimes \id$. If $\mathbb{B}=(\mathcal{B},\beta):\mathbb{T} \to \mathbb{T}'$ is a 1-morphism between trivializations and $\mathcal{J}:\mathcal{K} \to \mathcal{K}'$ is an isomorphism between bundle gerbes, we obtain a new 1-morphism 
\begin{equation*}
\mathbb{B}.\mathcal{J}: \mathbb{T}.\mathcal{K} \to \mathbb{T}'.\mathcal{K}'
\end{equation*}
consisting of the isomorphism $\mathcal{B} \otimes \pi^{*}\mathcal{J}: \mathcal{S} \otimes \pi^{*}\mathcal{K}\to \mathcal{S}' \otimes \pi^{*}\mathcal{K}'$ and of the transformation $\beta \otimes \id$. Finally, if $\varphi: \mathbb{B}\Rightarrow \mathbb{B}'$ is a 2-morphism between trivializations, and $\phi: \mathcal{J} \Rightarrow \mathcal{J}'$ is a transformation between isomorphisms of bundle gerbes over $M$, we have a new 2-morphism
\begin{equation*}
\varphi.\phi: \mathbb{B}.\mathcal{J} \Rightarrow \mathbb{B}'.\mathcal{J}'
\end{equation*} 
simply defined by $\varphi\otimes \pi^{*}\phi$. The compatibility condition for the transformation $\varphi \otimes \pi^{*}\phi$ is satisfied since $\phi$ drops out due to $\pi_1^{*}\pi^{*}\phi=\pi_2^{*}\pi^{*}\phi$ over $Y^{[2]}$.

Summarizing, the action of bundle gerbes on trivializations is a combination of the pullback $\pi^{*}$ and the tensor product of the monoidal 2-category of bundle gerbes. From this point of view, all the axioms of the action 2-functor
\erf{action} follow from those of the monoidal structure.
It is also immediately clear that a genuine action on isomorphism classes is induced. It remains to show that this action is free and transitive. 

To see that the action is free, assume that there exists a 1-morphism $\mathbb{T}.\mathcal{K} \to \mathbb{T}$ for $\mathbb{T}$ a trivialization of $\mathbb{G}$ and $\mathcal{K}$ an isomorphism. This implies
\begin{equation*}
\mathrm{DD}(\mathcal{S}) + \pi^{*}\mathrm{DD}(\mathcal{K}) = \mathrm{DD}(\mathcal{S})\text{,}
\end{equation*}
so that $\pi^{*}\mathrm{DD}(\mathcal{K})=0$. Since $\pi$ is a covering, $\mathrm{DD}(\mathcal{K})=0$. Thus, $\mathcal{K}$ is a trivial bundle gerbe up to isomorphism.  

To see that the action is transitive we infer that bundle gerbes form a 2-stack over smooth manifolds. The gluing property of this 2-stack has been shown in \cite[Proposition 6.7]{stevenson1}. We also  use the duality on the 2-groupoid of bundle gerbes (see Section \ref{bundlegerbes}).

Suppose $\mathbb{T}_1=(\mathcal{S}_1,\mathcal{A}_1,\sigma_1)$ and $\mathbb{T}_2=(\mathcal{S}_2,\mathcal{A}_2,\sigma_2)$ are trivializations of a bundle 2-gerbe $\mathbb{G}$. We will show that the bundle gerbe $\mathcal{G} := \mathcal{S}_1^{*} \otimes \mathcal{S}_2$ over $Y$ is a descent object for the 2-stack of bundle gerbes: there is an isomorphism $\mathcal{J}: \pi_2^{*}\mathcal{G} \to \pi_1^{*}\mathcal{G}$ of bundle gerbes over $Y^{[2]}$ and a 2-isomorphism 
\begin{equation*}
\varphi: \pi_{12}^{*}\mathcal{J} \circ \pi_{23}^{*}\mathcal{J} \Rightarrow \pi_{13}^{*}\mathcal{J}
\end{equation*}
over $Y^{[3]}$ which satisfies an associativity condition over $Y^{[4]}$. Then, the gluing property implies the existence of a bundle gerbe $\mathcal{K}$ over $M$, of an isomorphism $\mathcal{C}:\pi^{*}\mathcal{K} \to \mathcal{G}$ and of a transformation
\begin{equation*}
\gamma: \pi_1^{*}\mathcal{C} \Rightarrow \mathcal{J} \circ \pi_2^{*}\mathcal{C}
\end{equation*}
such that the diagram
\begin{equation}
\label{compgamma}
\alxydim{@=1.4cm}{\pi_{12}^{*}\mathcal{J} \circ \pi_2^{*}\mathcal{C} \ar@{=>}[r]^-{\id \circ \pi_{23}^{*}\gamma}& \pi_{12}^{*}\mathcal{J} \circ \pi_{23}^{*}\mathcal{J} \circ \pi_3^{*}\mathcal{C} \ar@{=>}[d]^{\varphi \circ \id}\\\pi_1^{*}\mathcal{C} \ar@{=>}[r]_-{\pi_{13}^{*}\gamma}  \ar@{=>}[u]^-{\pi_{12}^{*}\gamma} &  \pi_{13}^{*}\mathcal{J} \circ \pi_3^{*}\mathcal{C}}
\end{equation} 
is commutative. We will then finish the proof of the transitivity by showing that $(\mathcal{C},\gamma)$ gives rise to  a 1-morphism $\mathbb{T}_1.\mathcal{K} \to \mathbb{T}_2$. 

Let us first define the descent data $(\mathcal{J},\varphi)$ for $\mathcal{G}=\mathcal{S}_1^{*} \otimes \mathcal{S}_2$.
The isomorphism $\mathcal{J}$ is defined as the composition
\begin{equation*}
\alxydim{}{\pi_2(\mathcal{S}_1^{*} \otimes \mathcal{S}_2) \ar@{=}[r] & \pi_2^{*}\mathcal{S}_1^{*} \otimes \mathcal{I} \otimes \pi_2^{*}\mathcal{S}_2 \ar[d]^{\id \otimes \mathcal{D}_{\mathcal{P}}^{-1} \otimes \id} \\ & \pi_2^{*}\mathcal{S}_1^{*} \otimes \mathcal{P}^{*} \otimes \mathcal{P} \otimes \pi_2^{*}\mathcal{S}_2 \ar[d]^{\mathcal{A}_1^{*-1} \otimes \mathcal{A}_2} \\ & \pi_1^{*}\mathcal{S}_1^{*} \otimes   \pi_1^{*}\mathcal{S}_2 \ar@{=}[r] & \pi_1^{*}(\mathcal{S}_1^{*} \otimes \mathcal{S}_2)\text{.}}
\end{equation*}
The transformation $\varphi$ is defined using the transformations $\sigma_1$ and $\sigma_2$, namely as the composition
\begin{small}
\begin{equation*}
\alxydim{@C=0.55cm}{\pi_{12}^{*}\mathcal{J} \circ \pi_{23}^{*}\mathcal{J} \ar@{=>}[r] & (\pi_{12}^{*}\mathcal{A}_1^{*-1} \otimes \pi_{12}^{*}\mathcal{A}_2) \circ (\pi_{23}^{*}\mathcal{A}_1^{*-1} \otimes \pi_{23}^{*}\mathcal{A}_2) \circ \mathcal{D}^{-1}_{\pi_{12}^{*}\mathcal{P} \otimes \pi_{23}^{*}\mathcal{P}} \ar@{=>}[d]^{\sigma_1^{*-1} \otimes \sigma_2} \\ & (\pi_{13}^{*}\mathcal{A}_1^{*-1} \otimes \pi_{13}^{*}\mathcal{A}_2) \circ (\mathcal{M}^{*-1} \otimes \mathcal{M}) \circ \mathcal{D}^{-1}_{\pi_{12}^{*}\mathcal{P} \otimes \pi_{23}^{*}\mathcal{P}} \ar@{=>}[d]^{\text{naturality of $\mathcal{D}$ w.r.t. }\mathcal{M}} \\ & (\pi_{13}^{*}\mathcal{A}_1^{*-1} \otimes \pi_{13}^{*}\mathcal{A}_2) \circ \mathcal{D}^{-1}_{\pi_{13}^{*}\mathcal{P}} \ar@{=}[r] & \pi_{13}^{*}\mathcal{J}\text{.}}
\end{equation*}
\end{small}Here, the first arrow summarizes transformations that come from the monoidal structure on the 2-category of bundle gerbes, and that are used to commute tensor products with composition.
The associativity condition for $\varphi$  follows from the compatibility of $\sigma_1$ and $\sigma_2$ with the associator $\mu$ of $\mathbb{G}$ (see Figure \ref{compass} on page \pageref{compass}).

By the gluing axiom, we  now have a bundle gerbe $\mathcal{K}$, an isomorphism $\mathcal{C}$, and a transformation $\gamma$ as claimed above. We define an isomorphism $\mathcal{B}: \mathcal{S}_1 \otimes \pi^{*}\mathcal{K} \to \mathcal{S}_2$ as the composition
\begin{equation*}
\alxydim{@C=1.4cm}{\mathcal{S}_1 \otimes \pi^{*}\mathcal{K} \ar[r]^-{\id \otimes \mathcal{C}} & \mathcal{S}_1 \otimes \mathcal{S}_1^{*}\otimes \mathcal{S}_2 \ar[r]^-{\mathcal{D}_{\mathcal{S}^{*}_1} \otimes \id} & \mathcal{I} \otimes \mathcal{S}_2 = \mathcal{S}_2\text{.}}
\end{equation*}
By a similar argument one can produce a transformation
\begin{equation*}
\alxydim{@C=1.4cm@R=1.4cm}{\mathcal{P} \otimes \pi_2^{*}\mathcal{S}_1 \otimes \pi_2^{*}\pi^{*}\mathcal{K} \ar[d]_{\id \otimes \pi_2^{*}\mathcal{B}} \ar[r]^-{\mathcal{A}_1 \otimes \id} & \pi_1^{*}\mathcal{S}_1 \otimes \pi_1^{*}\pi^{*}\mathcal{K} \ar@{=>}[dl]|*+{\beta} \ar[d]^{\pi_1^{*}\mathcal{B}} \\ \mathcal{P} \otimes \pi_2^{*}\mathcal{S}_2 \ar[r]_{\mathcal{A}_2} & \pi_1^{*}\mathcal{S}_2}
\end{equation*}
using $\gamma$. More precisely, $\beta$ is defined as
\begin{small}
\begin{eqnarray*}
&\xymatrix{\pi_1^{*}\mathcal{B} \circ (\mathcal{A}_1 \otimes \id) \ar@{=}[r] & (\mathcal{D}_{\pi_1^{*}\mathcal{S}^{*}_1} \otimes \id) \circ (\mathcal{A}_1 \otimes \id^{\otimes 2}) \circ (\id^{\otimes2} \otimes \pi_1^{*}\mathcal{C})}\hspace{3cm}&
\\[-\bigskipamount]
&\xymatrix@C=5cm{&\ar@{=>}[d]^{\id \circ \id\circ (\id^{\otimes2} \otimes \gamma)}&\\&~&}& \\[-\medskipamount]
&(\mathcal{D}_{\pi_1^{*}\mathcal{S}^{*}_1} \otimes \id) \circ (\mathcal{A}_1 \otimes \id^{\otimes 2}) \circ (\id^{\otimes2} \otimes \mathcal{J}) \circ (\id^{\otimes2} \otimes \pi_2^{*}\mathcal{C}) &
\\[-\bigskipamount]
&\xymatrix@C=5cm{&\ar@{=>}[d]^{\text{Def. of $\mathcal{J}$}}& \\ &~&}&
\\[-\medskipamount]
&(\mathcal{D}_{\pi_1^{*}\mathcal{S}_1^{*}} \otimes \id) \circ (\mathcal{A}_1 \otimes \mathcal{A}_1^{*-1} \otimes \id) \circ (\id^{\otimes4} \otimes \mathcal{A}_{2}) \circ (\id^{\otimes3} \otimes \mathcal{D}^{-1}_{\mathcal{P}} \otimes \id) \circ (\id^{\otimes2} \otimes \pi_2^{*}\mathcal{C})&
\\[-\bigskipamount]
&\xymatrix@C=5cm{&\ar@{=>}[d]^{\text{naturality of $\mathcal{D}$ applied to $\mathcal{A}_1$}}&\\ &~&}&
\\[-\medskipamount]
&(\mathcal{D}_{\pi_2^{*}\mathcal{S}_1} \otimes \mathcal{D}_{\mathcal{P}^{*}})  \circ (\id^{\otimes4} \otimes \mathcal{A}_{2})\circ (\id^{\otimes3} \otimes \mathcal{D}^{-1}_{\mathcal{P}} \otimes \id) \circ (\id^{\otimes2} \otimes \pi_2^{*}\mathcal{C})&
\\[-\bigskipamount]
&\xymatrix@C=5cm{&\ar@{=>}[d]^{\text{compatibility between $\otimes$ and $\circ$}}&\\&~&}&
\\[-\medskipamount]
& (\mathcal{D}_{\mathcal{P}^{*}} \otimes \id)  \circ (\id^{\otimes 2} \otimes \mathcal{A}_{2}) \circ (\id^{\otimes3} \otimes \mathcal{D}^{-1}_{\mathcal{P}} \otimes \id) \circ (\id \otimes \mathcal{D}_{\pi_2^{*}\mathcal{S}_1^{*}} \otimes \id) \circ (\id^{\otimes2} \otimes \pi_2^{*}\mathcal{C})&
\\[-\bigskipamount]
&\xymatrix@C=5cm{&\ar@{=>}[d]^{\text{zigzag for $\mathcal{D}_{\mathcal{P}}$}} &\\&~&}&
\\[-\medskipamount]
&\hspace{2cm}\xymatrix{\mathcal{A}_{2} \circ(\id \otimes \mathcal{D}_{\pi_2^{*}\mathcal{S}^{*}_1} \otimes \id)
                       \circ (\id^{\otimes2} \otimes \pi_2^{*}\mathcal{C}) \ar@{=}[r] &\mathcal{A}_2
                       \circ (\id \otimes \pi_2^{*}\mathcal{B})\text{.}}&
\end{eqnarray*}
\end{small}Finally, one can deduce from the commutativity of \erf{compgamma} and the definition of the transformation $\varphi$ that $\beta$ is compatible with the transformations $\sigma_1$ and $\sigma_2$ in the sense of Figure \ref{compmorph}. Hence, $(\mathcal{B},\beta)$ is a 1-morphism from $\mathbb{T}_1.\mathcal{K}$ to $\mathbb{T}_2$.

\begin{remark}
\label{actgens}
All constructions and results of this section can straightforwardly be generalized to bundle gerbes \emph{with connection} in the sense of Remark \ref{catgens}. The module structure is then a strict 2-functor
\begin{equation*}
\triv {\mathbb{G},\nabla} \times \ugrbcon M \to \triv {\mathbb{G},\nabla}\text{.}
\end{equation*}
In the proof that the action is free, one substitutes characteristic degree three classes \erf{grbconclass} in differential cohomology  
for the  Dixmier-Douady classes.
In the  proof that the action is transitive, one uses that bundle gerbes \emph{with connection} also form a 2-stack: the gluing condition  is shown in \cite[Section 3.3]{nikolaus1}.
\end{remark}

\subsection{Lemma \ref{geomstrex}: Existence of Trivializations with Connection}

\label{trivcon}

We show that every bundle 2-gerbe $\mathbb{G}$ with connection and vanishing characteristic class in $\h^4(M,\Z)$ admits a trivialization with compatible connection. For the proof we use the fact that bundle 2-gerbes are classified up to isomorphism by degree four differential cohomology $\hat\h^4(M,\Z)$. This cohomology group fits into the exact sequence 
\begin{equation}
\label{delex}
\alxydim{@C=1.2cm}{\Omega^3(M) \ar[r] & \hat\h^4(M,\Z) \ar[r]^-{\mathrm{CC}} & \h^4(M,\Z) \ar[r] & 0\text{.}}
\end{equation}
Suppose $\mathbb{G}$ is a bundle 2-gerbe with connection $\nabla$ and with vanishing characteristic  class $\mathrm{CC}(\mathbb{G})$. By  exactness of \erf{delex}, it is isomorphic to a certain bundle 2-gerbe $\mathbb{I}$ with connection $\nabla_H$ defined by a 3-form $H \in \Omega^3(M)$. We will show that such an isomorphism is precisely a trivialization of $\mathbb{G}$ with connection compatible with $\nabla$. This  proves Lemma \ref{geomstrex}.

Let us first describe the bundle 2-gerbe $\mathbb{I}$ and the connection $\nabla_H$ which is associated to any 3-form $H$ on $M$. The covering of $\mathbb{I}$ is the identity $\id:M \to M$ whose fibre products we can identify with $M$ itself. Its bundle gerbe is the trivial bundle gerbe $\mathcal{I}$, which is the tensor unit of the monoidal 2-category $\ugrb M$. Its product is the identity $\id: \mathcal{I} \otimes \mathcal{I} \to \mathcal{I}$, and its associator is also the identity. The only non-trivial information is the connection $\nabla_{H}$. It consists simply of the 3-form $H$; the bundle gerbe $\mathcal{I}$ and the product of $\mathbb{I}$ carry  trivial connections.

Next we give a brief definition of an isomorphism between two bundle 2-gerbes $\mathbb{G}_1$ and $\mathbb{G}_2$. It is a straightforward generalization of the notion of an isomorphism between bundle gerbes (see Definition \ref{defiso}). An \emph{isomorphism} $\mathbb{A}:\mathbb{G}_1 \to \mathbb{G}_2$ consists of a bundle gerbe $\mathcal{S}$ over a covering $\zeta: Z \to Y_1 \times_M Y_2$, an isomorphism
\begin{equation*}
\mathcal{A}: \mathcal{P}_1 \otimes \zeta_2^{*}\mathcal{S} \to \zeta_1^{*}\mathcal{S} \otimes \mathcal{P}_2
\end{equation*}
between bundle gerbes over $Z \times_M Z$, 
and a transformation $\sigma$ which expresses the compatibility between $\mathcal{A}$ and the products $\mathcal{M}_1$ and $\mathcal{M}_2$. This transformation
has to satisfy an evident coherence condition involving the associators $\mu_1$ and $\mu_2$ of the two bundle 2-gerbes. 

If the bundle 2-gerbes $\mathbb{G}_1$ and $\mathbb{G}_2$ are equipped with connections, we say that a \emph{compatible connection} on the isomorphism $\mathbb{A}$ is a connection on the bundle gerbe $\mathcal{S}$ of curvature
\begin{equation*}
\mathrm{curv}(S) = \pi_2^{*}B_2 - \pi_1^{*}B_1\text{,}
\end{equation*}
and a compatible connection on the isomorphism $\mathcal{A}$ such that $\sigma$ is connection-preserving. 

The claimed relation to differential cohomology (realized by Deligne cohomology) is established in \cite[Proposition 4.2]{johnson1} in terms of a bijection
\begin{equation*}
\bigset{5.7cm}{Bundle 2-gerbes over $M$ with connection, up to isomorphisms with compatible connection} \cong \hat\h^4(M,\Z)\text{.}
\end{equation*} 

Comparing the definition of an isomorphism between bundle 2-gerbes and Definition \ref{deftriv} of a trivialization makes it obvious that a trivialization $\mathbb{T}$ of $\mathbb{G}$ is the same thing as an isomorphism $\mathbb{T}:\mathbb{G} \to \mathbb{I}$. This coincidence generalizes to a setup with compatible connections: there is a bijection
\begin{equation}
\label{trivasmorph}
\bigset{5.5cm}{Isomorphisms $\mathbb{T}:\mathbb{G} \to \mathbb{I}$ with connection compatible with $\nabla$ and $\nabla_H$ for some 3-form $H$} \cong \bigset{3.4cm}{Trivializations of $\mathbb{G}$ with connection compatible with $\nabla$}
\end{equation}
 for every bundle 2-gerbe $\mathbb{G}$ and any connection $\nabla$ on $\mathbb{G}$.

\subsection{Lemma \ref{conind}: Connections on Trivializations pull back}

\label{conpullback}

Let $\mathbb{G}$ be a bundle 2-gerbe with connection, and let $\mathbb{T}=(\mathcal{S},\mathcal{A},\sigma)$ and $\mathbb{T}'=(\mathcal{S}',\mathcal{A}',\sigma')$ be trivializations of $\mathbb{G}$. We prove that one can pull back a compatible  connection $\babla'$ on $\mathbb{T}'$ along any 1-morphism $\mathbb{B}:\mathbb{T} \to \mathbb{T}'$ to a compatible connection on $\mathbb{T}$.

We recall that a compatible connection on $\mathbb{T}'$ is a pair $\babla'=(\nabla',\omega')$ with $\nabla'$ a connection on $\mathcal{S}'$ and $\omega'$ a connection on $\mathcal{A}'$, such that $\sigma'$ is connection-preserving.

By Lemma \ref{conpull}, there exists a connection $\kappa$ on the isomorphism $\mathcal{B}:\mathcal{S} \to \mathcal{S}'$ and a connection  on $\mathcal{S}$ such that $\kappa$ is compatible. Next we need a connection on the isomorphism $\mathcal{A}$. We look at the transformation
\begin{equation}
\label{compconpulltrans}
\alxydim{@C=1.2cm}{\mathcal{A} \ar@{=>}[r]^-{\rho^{-1}_{\mathcal{A}}} & \id \circ \mathcal{A} \ar@{=>}[r]^-{\pi_1^{*}i_r \otimes \id} & \pi_1^{*}\mathcal{B} \circ \pi_1^{*}\mathcal{B}^{-1} \circ \mathcal{A}\ar@{=>}[r]^-{\id \circ \beta} & \pi_1^{*}\mathcal{B}^{-1} \circ \mathcal{A}' \circ \pi_2^{*}\mathcal{B}\text{.}}
\end{equation}
Here we have used the  transformation $\rho_{\mathcal{A}}: \id \circ \mathcal{A} \Rightarrow \mathcal{A}$ which belongs to the structure of the 2-groupoid of bundle gerbes \cite[Section 1.2]{waldorf1}, and the  transformation $i_r: \id \Rightarrow \mathcal{B} \circ \mathcal{B}^{-1}$ which expresses the invertibility of the isomorphism $\mathcal{B}$ \cite[Section 1.3]{waldorf1}. Notice that the target isomorphism of \erf{compconpulltrans} is equipped with a compatible connection. Thus, by Lemma \ref{compconpullback}, this connection pulls back to a compatible connection on $\mathcal{A}$. Furthermore, since the transformations \erf{compconpulltrans}, $\rho_{\mathcal{A}}$ and $i_r$ are connection-preserving,  the transformation $\id \circ \beta$ is also connection-preserving. This implies in turn that $\beta$ itself is connection-preserving.

It remains to check that the transformation $\sigma$ preserves connections. In order to see this, consider the commutative diagram of Figure \ref{compmorph}, which expresses the compatibility between $\sigma$, $\sigma'$ and $\beta$. Since all transformations that appear in this diagram are invertible, we can rearrange it as an equation
\begin{equation*}
\id \circ \sigma = (\pi_{13}^{*}\beta \circ \id)^{-1} \bullet (\sigma' \circ \id) \bullet (\id \circ \pi_{23}^{*}\beta) \bullet (\pi_{12}^{*}\beta \circ \id)^{-1}\text{,} 
\end{equation*} 
where $\bullet$ denotes the vertical composition of transformations. Now, since the right hand side of this equation is a connection-preserving transformation,  $\id \circ \sigma$ is also connection-preserving. Just as above, it follows that $\sigma$ is connection-preserving.

\subsection{Lemma \ref{lemaction1} (a): Well-definedness of the Action on compatible Connections}

\label{action1}

We prove that the action of the vector space $V_{\mathbb{T}}$ on the connections on a trivialization $\mathbb{T}$ of a bundle 2-gerbe $\mathbb{G}$ with connection takes compatible connections to compatible connections. 

The bundle 2-gerbe $\mathbb{G}$ has a covering $\pi:Y \to M$ and a bundle gerbe $\mathcal{P}$ over $Y^{[2]}$. The bundle gerbe $\mathcal{P}$  in turn has a covering $\chi:X \to Y^{[2]}$. The trivialization $\mathbb{T}$ has a bundle gerbe $\mathcal{S}$ over $Y$, an isomorphism $\mathcal{A}$ over $Y^{[2]}$ and a transformation $\sigma$. Its bundle gerbe $\mathcal{S}$ has a covering $\omega: W \to Y$. Expanding the definitions of tensor products and isomorphisms between bundle gerbes, the isomorphism $\mathcal{A}: \mathcal{P} \otimes \pi_2^{*}\mathcal{S} \to \pi_1^{*}\mathcal{S}$ comes with a covering 
\begin{equation*}
\zeta: Z \to X \times_{Y^{[2]}} (W \times_M W)\text{.}
\end{equation*}
The projections $x: Z \to X$ and $p:Z \to W \times_M W$ are again coverings. By construction, there is a commutative diagram:
\begin{equation}
\label{comdiag1}
\alxydim{@C=0.4cm@R=0.4cm}{&Z \ar[dd]^{p} && \\&&&\\ &W \times_M W \ar[dd]|{\omega \times \omega} \ar[dr]^-{p_1} \ar[dl]_-{p_2}&&\\W \ar[dd]_{\omega} &&W \ar[dd]^{\omega}&\\&Y^{[2]} \ar[dr]|{\pi_2} \ar[dl]|{\pi_1}&&\\Y \ar[dr]_{\pi}&&Y \ar[dl]^{\pi}&\\&M&&}
\end{equation}

We keep a connection on $\mathbb{G}$ fixed, and assume a compatible connection $\babla=(C,\omega,\kappa)$
 on $\mathbb{T}$, where $(C,\omega)$ is a connection on $\mathcal{S}$. Let $(\psi,\rho,\varphi) \in V_{\mathbb{T}}$ represent an element in the vector space that  acts on the set of  connections of $\mathbb{T}$. Its action on $\babla$ has been defined in Section \ref{exclass} to result in $\babla'=(C',\omega',\kappa')$ with
\begin{eqnarray}
\nonumber C'&=&C + \mathrm{d}\varphi- \omega^{*}(\mathrm{d}\rho - \pi^{*}\psi)\text{,} \\
\label{change}\omega' &=& \omega + \delta_{\omega}\varphi\text{,}\\
\nonumber \kappa'&=& \kappa + \epsilon\quad\text{ with }\quad \varepsilon := p^{*}(\delta_{\pi\circ\omega}(\varphi-\omega^{*}\rho))\text{.}
\end{eqnarray}
Here we have used the notation $\delta_{\pi\circ\omega}$, $\delta_{\omega}$ for the alternating sum over pullbacks along a surjective submersion, as explained in Lemma \ref{exact}.

Let us first check that $\kappa'$ is  a compatible connection on $\mathcal{A}$. The first condition is the equation
\begin{equation}
\label{1}
\mathrm{curv}(\kappa') = p^{*}p_2^{*}C' - (x^{*}C_{\mathcal{P}} + p^{*}p_1^{*}C')
\end{equation} 
of 2-forms over $Z$, where $C_{\mathcal{P}}$ is the 2-form of the connection on $\mathcal{P}$. This equation can be verified using the fact that $\kappa$ was assumed to be compatible, and using the commutative diagram \erf{comdiag1} to sort out the various pullbacks. The second condition for $\kappa'$  is that the isomorphism
\begin{equation*}
\alpha: (x^{*}P_{\mathcal{P}} \otimes w_1^{*}P_{\mathcal{S}}) \otimes \zeta_2^{*}Q \to \zeta_1^{*}Q \otimes w_2^{*}P_{\mathcal{S}}
\end{equation*}
of principal $\ueins$-bundles over $Z \times_{Y^{[2]}} Z$ is  connection-preserving. Here, $\zeta_1,\zeta_2: Z \times_{Y^{[2]}} Z \to Z$ denote the two projections, $x: Z \times_{Y^{[2]}} Z \to X \times_{Y^{[2]}}X$ is just the map $x$ from above on each factor, and $w_i:Z \times_{Y^{[2]}} Z \to W \times_Y W$ is the map  $p_i \circ p: Z \to W$  on each factor. There is a commutative diagram exploiting the various relations between these maps ($k=1,2$):
\begin{equation}
\label{diag2}
\alxydim{@R=0.7cm@C=0.1cm}{&&Z \times_{Y^{[2]}} Z \ar[ddll]_{\zeta_1} \ar[ddrr]^{\zeta_2} \ar[dd]|{w_k}&&\\&&&&\\Z \ar[dr]_{p_k} &&W \times_Y W \ar[dl]|{\omega_1}\ar[dr]|{\omega_2}&& Z \ar[dl]^{p_k} \\ &W \ar[dr]_{\omega}&&W \ar[dl]^{\omega}&\\&&Y&&}
\end{equation}
To check that $\alpha$ is connection-preserving, one verifies that the connections on both sides change by the same 1-form on $Z \times_{Y^{[2]}}Z$ under the action \erf{change}. On the left hand side, this is $w_1^{*}\delta_{\omega}\varphi + \zeta_2^{*} \epsilon$. On the right hand side, it is $\zeta_1^{*} \varepsilon\ + w_2^{*}\delta_{\omega}\varphi$. Using diagram \erf{diag2}, it is straightforward to check that these 1-forms coincide. 

The remaining check is that the transformation $\sigma$ is still connection-preserving. We note that $\sigma$ is some isomorphism of principal $\ueins$-bundles over a smooth manifold $V$. There are projections $k_s$ and $k_t$ into the covering spaces of the source isomorphism $\pi_{12}^{*}\mathcal{A} \circ (\id_{\pi_{12}^{*}\mathcal{P}} \otimes \pi_{23}^{*}\mathcal{A})$ and the target isomorphism $\pi_{13}^{*}\mathcal{A} \circ (\mathcal{M} \otimes \id_{\pi_{3}^{*}\mathcal{S}})$. We can safely ignore the contributions of the isomorphism $\mathcal{M}$ and of the identity isomorphism $\id_{\pi_{12}^{*}\mathcal{P}}$ in the following discussion, since the connections on $\mathcal{M}$ and on $\mathcal{P}$ did not change under our action. What we must not ignore is the identity isomorphism $\id_{\pi_3^{*}\mathcal{S}}$: its connection changes with the connection on $\mathcal{S}$!

After these premises, the two projections are $k_s: V \to Z \times_W Z$ and $k_t: V \to (W \times_Y W) \times_W Z$, with $Z$ the covering space of the isomorphism $\mathcal{A}$ and $W \times_Y W$ the covering space of the identity isomorphism. We claim that the following diagrams are commutative by construction:
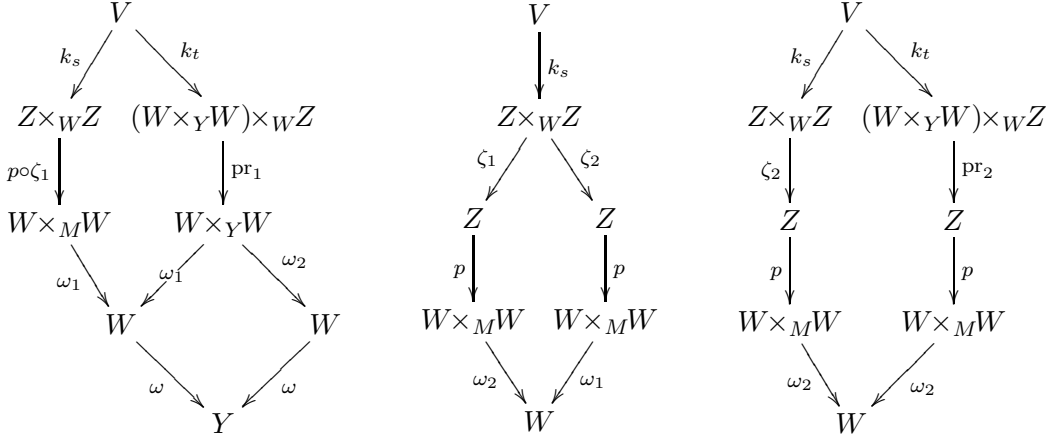
\begin{small}
\begin{equation*}
\alxydim{@C=-0.3cm}{&V\ar[dl]_{k_s} \ar[dr]^{k_t}& \\ Z \ttimes W Z\ar[d]_{p\circ\zeta_1} && (W \ttimes Y W) \ttimes W Z\ar[d]^{\mathrm{pr}_1} \\ W \ttimes M W \ar[dr]_{\omega_1} && W \ttimes Y W \ar[dr]^{\omega_2} \ar[dl]|{\omega_1} \\   & W \ar[dr]_{\omega} && W \ar[dl]^{\omega} \\ &&Y}
\quad\quad
\alxydim{@C=-0.6cm}{&V \ar[d]^{k_s}& \\ &Z \ttimes W Z \ar[dl]_{\zeta_1} \ar[dr]^{\zeta_2}&\\Z \ar[d]_{p} && Z \ar[d]^{p} \\ W \ttimes M W \ar[dr]_{\omega_2} && W \ttimes M W \ar[dl]^{\omega_1}\\&W&}
\quad\quad
\alxydim{@C=-0.3cm}{&V \ar[dr]^{k_t} \ar[dl]_{k_s}&\\Z \ttimes W Z \ar[d]_{\zeta_2} && (W \ttimes Y W) \ttimes W Z \ar[d]^{\mathrm{pr}_2}\\Z \ar[d]_{p} && Z \ar[d]^{p}\\W \ttimes M W \ar[dr]_{\omega_2} && W \ttimes M W \ar[dl]^{\omega_2}\\ &W&}
\end{equation*} 
\end{small}
We compute the changes in the connections on the target and on the source isomorphism of $\sigma$. These are, respectively, the 1-forms
\begin{equation}
\label{coincone}
k_t^{*}(\mathrm{pr}_1^{*}\delta\varphi + \mathrm{pr}_2^{*}\varepsilon)
\quad\text{ and }\quad
k_s^{*}(\zeta_1^{*}\varepsilon + \zeta_2^{*}\varepsilon)
\end{equation}
over $V$. Using the diagrams above it is a straightforward calculation to check that these 1-forms coincide. Thus, $\sigma$ is a connection-preserving transformation. 

Summarizing, we have shown that the action of an element of $V_{\mathbb{T}}$ on a compatible connection $\babla$ on $\mathbb{T}$ is again a compatible connection $\babla'$.

\subsection{Lemma \ref{lemaction2} (b): The Action on compatible Connections is free and transitive}

\label{action2}

We prove that the action of the vector space $V_{\mathbb{T}}$ on the compatible connections of a bundle 2-gerbe $\mathbb{G}$ with connection is free and transitive.

Showing that the action is free is the easy part. We assume that an element $(\psi,\rho,\varphi)\in V_{\mathbb{T}}$ acts trivially on a connection $\babla=(C,\omega,\kappa)$ on a trivialization $\mathbb{T}=(\mathcal{S},\mathcal{A},\sigma)$. That means that the 1-form $\varepsilon = p^{*}\delta_{\pi\circ\omega}(\varphi-\omega^{*}\rho)$ vanishes and that $(\eta,\varphi) \in V_{\mathcal{S}}$ with $\eta = \mathrm{d}\rho - \pi^{*}\psi$ vanishes in $V_{\mathcal{S}}$. 

Since $p$ is a covering, the vanishing of $\varepsilon$ implies that already $\delta_{\pi\circ\omega}(\varphi-\omega^{*}\rho)=0$. By Lemma \ref{exact}, this implies the existence of a 1-form $\chi \in\Omega^1(M)$ such that (I) $\omega^{*}\pi^{*}\chi= \omega^{*}\rho - \varphi$. The vanishing of $(\eta,\varphi)$ implies the existence of a 1-form $\nu\in\Omega^1(Y)$ such that (II) $\eta=\mathrm{d}\nu$ and (III) $\varphi = \omega^{*}\nu$. (I) and (III) imply (IV) $\rho = \pi^{*}\chi + \nu$. (II) and (IV) imply (V) $\mathrm{d}\chi=\psi$.  Equations (V), (IV) and (III) show that $(\psi,\rho,\varphi)$ lies in the subspace $U$ we  divide out in Proposition \ref{trivconaff}. 

Now we prove that the action is transitive. We assume that $\babla=(C,\varphi,\kappa)$ and $\babla'=(C',\varphi',\kappa')$ are two compatible connections on a trivialization $\mathbb{T}$. From Lemmata \ref{affine} and  \ref{affineiso} we obtain $\epsilon \in \Omega^1(Z)$, $\varphi \in \Omega^1(W)$  and $\eta\in\Omega^2(Y)$ such that
\begin{equation*}
C'=C + \mathrm{d}\varphi- \omega^{*}\eta
\quad\text{, }\quad
\omega' = \omega + \delta_{\omega}\varphi
\quad\text{ and }\quad
\kappa'= \kappa + \epsilon\text{.}
\end{equation*}
First we consider the 1-form 
\begin{equation*}
\tilde\rho := \epsilon - p_2^{*}\varphi + p_1^{*}\varphi \in \Omega^1(Z)
\end{equation*}
with $p_k: Z \to W$ the projections from Section \ref{action1}. We denote the evident projection to the base space of $\mathcal{A}$ by $\ell:Z \to Y^{[2]}$. Using the identity 
\begin{equation*}
\zeta_2^{*} \epsilon=\zeta_1^{*} \varepsilon\ + w_2^{*}\delta_{\omega}\varphi-w_1^{*}\delta_{\omega}\varphi
\end{equation*}
that we have derived in Section \ref{action1}, it is straightforward to check that  
\begin{equation*}
\delta_{\ell}\tilde\rho =\zeta_2^{*}\tilde\rho - \zeta_1^{*}\tilde\rho = 0\text{,}
\end{equation*}
so that by Lemma \ref{exact}, there exists a 1-form $\rho'\in\Omega^1(Y^{[2]})$ such that $\ell^{*}\rho' = \tilde\rho$.

Now we denote by $V$ the covering space of the transformation $\sigma$, and we denote by $k:V \to Y^{[3]}$ the evident projection to the base space 
of the involved bundle gerbes.  Note that the following diagrams are commutative:
\begin{equation*}
\alxydim{}{V \ar[r]^{\zeta_2 \circ k_s} \ar[d]_{k} & Z \ar[d]^{\ell}\\ Y^{[3]} \ar[r]_{\pi_{21}} & Y^{[2]}}
\quad\quad
\alxydim{}{V \ar[r]^{\zeta_1 \circ k_s} \ar[d]_{k} & Z \ar[d]^{\ell}\\ Y^{[3]} \ar[r]_{\pi_{32}} & Y^{[2]}}
\quad\quad
\alxydim{}{V \ar[r]^{\mathrm{pr}_2 \circ k_t} \ar[d]_{k} & Z \ar[d]^{\ell}\\ Y^{[3]} \ar[r]_{\pi_{31}} & Y^{[2]}}
\quad\quad
\alxydim{@C=0.6cm}{V \ar[r]^{\mathrm{pr}_2 \circ k_t} \ar[d]_{\mathrm{pr}_1\circ k_t} & Z \ar[d]^{p_1}\\ W \ttimes Y W \ar[r]_-{\omega_2} & W\text{.}}
\end{equation*}
Using these diagrams and the coincidence of the 1-forms \erf{coincone}, one readily verifies
\begin{equation*}
k^{*}(\pi_{21}^{*}\rho' - \pi_{31}^{*}\rho' + \pi_{32}^{*}\rho') = k_s^{*}\zeta_2^{*}\tilde\rho - k_t^{*}\mathrm{pr}_2^{*}\tilde\rho  +k_s^{*}\zeta_1^{*}\tilde\rho = 0\text{,}
\end{equation*}
so that again by Lemma \ref{exact}, there exists a 1-form $\rho\in\Omega^1(Y)$ such that $\delta_{\pi}\rho=\rho'$.

Finally we consider $\psi' := \mathrm{d}\rho - \eta \in \Omega^2(Y)$. We compute
\begin{equation*}
\ell^{*}(\pi_2^{*}\psi' - \pi_1^{*}\psi') = 0\text{,}
\end{equation*}
using the remaining condition \erf{1}. Thus, we find a 1-form $\psi \in \Omega^2(M)$  such that $\pi^{*}\psi = \psi'$. Tracing all definitions back, we see that acting with $(\psi,\rho,\varepsilon)$ on $\babla$ we obtain $\babla'$.

\kobib{../../bibliothek/tex}

\end{document}